\documentclass[11pt]{article}

\usepackage[a4paper,margin=1.1in]{geometry}
\usepackage{amsmath,amssymb,amsthm,mathtools}
\usepackage[T1]{fontenc}
\usepackage{lmodern}
\usepackage{microtype}
\usepackage{enumitem}
\usepackage{hyperref}
\usepackage{aliascnt}
\usepackage[nameinlink,noabbrev,capitalise]{cleveref}

\numberwithin{equation}{section}

\hypersetup{
  hidelinks
}

\newtheorem{theorem}{Theorem}[section]

\newaliascnt{proposition}{theorem}
\newtheorem{proposition}[proposition]{Proposition}
\aliascntresetthe{proposition}

\newaliascnt{lemma}{theorem}
\newtheorem{lemma}[lemma]{Lemma}
\aliascntresetthe{lemma}

\newaliascnt{corollary}{theorem}
\newtheorem{corollary}[corollary]{Corollary}
\aliascntresetthe{corollary}

\theoremstyle{remark}
\newaliascnt{remark}{theorem}
\newtheorem{remark}[remark]{Remark}
\aliascntresetthe{remark}

\theoremstyle{plain}

\crefname{equation}{equation}{equations}
\Crefname{equation}{Equation}{Equations}

\crefname{theorem}{Theorem}{Theorems}
\Crefname{theorem}{Theorem}{Theorems}

\crefname{proposition}{Proposition}{Propositions}
\Crefname{proposition}{Proposition}{Propositions}

\crefname{lemma}{Lemma}{Lemmas}
\Crefname{lemma}{Lemma}{Lemmas}

\crefname{corollary}{Corollary}{Corollaries}
\Crefname{corollary}{Corollary}{Corollaries}

\crefname{remark}{Remark}{Remarks}
\Crefname{remark}{Remark}{Remarks}

\newcommand{\PP}{\mathbf P}
\newcommand{\ZZ}{\mathbf Z}
\newcommand{\Aut}{\operatorname{Aut}}
\newcommand{\AutL}{\operatorname{Aut}_L}
\newcommand{\Gr}{\operatorname{Gr}}
\newcommand{\Sym}{\operatorname{Sym}}
\newcommand{\Cl}{\operatorname{Cl}}

\newcommand{\GL}{\operatorname{GL}}
\newcommand{\PGL}{\operatorname{PGL}}
\newcommand{\Img}{\operatorname{Im}}
\newcommand{\Ker}{\operatorname{Ker}}
\newcommand{\Span}{\operatorname{Span}}
\newcommand{\rk}{\operatorname{rk}}
\newcommand{\id}{\mathrm{id}}
\newcommand{\Pic}{\operatorname{Pic}}
\newcommand{\Hom}{\operatorname{Hom}}

\title{Generic Triviality of Linear Automorphisms of Complete Intersections}

\author{
Zhaoning Cui$^{\alpha,\beta}$\\[0.5em]
\small $^\alpha$ Institute for Theoretical and Mathematical Physics, Lomonosov Moscow State University\\
\small $^\beta$ Department of Mechanics and Mathematics, Lomonosov Moscow State University}

\date{}

\begin{document}

\maketitle

\begin{abstract}
We prove that a general smooth complete intersection of positive dimension and
codimension at least \(2\) over an algebraically closed field has trivial linear
automorphism group, except in the case of two quadrics. The proof uses incidence
varieties and treats separately prime-to-characteristic and unipotent automorphisms.
\end{abstract}

\setcounter{tocdepth}{2}
\setcounter{secnumdepth}{2}

\tableofcontents

\section{Introduction}
This article is about linear automorphism groups of complete intersections. In the hypersurface case, Matsumura and Monsky proved that the linear
automorphism group of every smooth hypersurface
\(Y\subset \PP^N_k\) of degree \(d\ge 3\), with \(N\ge 3\), is finite,
and is trivial for a generic hypersurface of the same degree
\cite[Theorems~1 and~5]{MatsumuraMonsky1963}.

Later, Chen, Pan, and Zhang proved that, for a general smooth complete
intersection of dimension at least \(2\) and multidegree different from
\((2,2)\), the linear automorphism group is trivial in characteristic \(0\),
while in characteristic \(p>0\) it is a finite \(p\)-group
\cite[Theorem~1.3(1)]{ChenPanZhang2024}.
They also determined the linear automorphism group of a general intersection
of two quadrics in characteristic different from \(2\) is not trivial
\cite[Theorem~1.3(2)]{ChenPanZhang2024}. More recently, Lyu and Zhang developed a monodromy-theoretic criterion
which reduces generic triviality in a family to the exclusion of a possible
involution, and applied it to establish generic triviality for several broad
classes of complete intersections in characteristic different from \(2\)
\cite[Theorems~1.2 and~2.1]{LyuZhang2024}.

Motivated by these developments, we study the remaining problems from a direct
incidence-theoretic point of view.

Let \(k\) be an algebraically closed field, and let \(X\subset \PP^n_k\) be a smooth
complete intersection of codimension \(c<n\). We write \(\Aut_L(X)\) for linear automorphism group, that is the subgroup of
\(\Aut(X)\) induced by projective linear transformations of \(\PP^n\).

The goal of this paper is to prove the following result.

\subsection{Main result}

\begin{theorem}\label{thm:main}
Let \(k\) be an algebraically closed field.

Let \(X\subset \PP^n_k\) be a general smooth complete intersection of multidegree
\[
2\le d_1\le \cdots \le d_c,\qquad c\ge 2,\qquad n-c\ge 1.
\]
Assume \((d_1,\dots,d_c)\neq (2,2)\). Then \(\Aut_L(X)=\{1\}\).
\end{theorem}

The exception \((2,2)\) is genuine. In characteristic different from \(2\), a
general pencil of quadrics is diagonal and coordinate sign changes give nontrivial
linear automorphisms. In characteristic \(2\), see Dolgachev--Duncan for
\(X\subset \PP^{2m}\) and Zhang for \(X\subset \PP^{2m+1}\)
\cite[Theorems~1.6 and~7.4]{DolgachevDuncan2018},
\cite[Corollary~2.8, Theorem~1.3 and Corollary~4.5]{YangZhang2026}.

\begin{corollary}\label{cor:lin}
Let \(X\subset \PP^n_k\) be a smooth complete intersection. Then
\(\Aut(X)=\Aut_L(X)\) in the following cases:
\(
\dim X\ge 3,
\)
and
\(
\dim X=2
 \text{ with } 
\sum_{i=1}^c d_i\neq n+1.
\)
The remaining surface case is the \(K3\) case
\(\sum_{i=1}^c d_i=n+1\), while in dimension \(1\) the comparison between
\(\Aut(X)\) and \(\Aut_L(X)\) is more subtle.
\end{corollary}

\subsection{Method of proof}

Our proof is based on a finite incidence-theoretic reduction, following the incidence
construction of Lyu--Zhang \cite[Proposition~2.15, pp.~16--19]{LyuZhang2024}. Let
\(T^\circ\) be the parameter space of smooth complete intersections of the prescribed
multidegree. First, using Benoist's theorem together with the finiteness
result for relative automorphism groups recalled by Lyu--Zhang, we reduce the problem to
excluding automorphisms of prime order; see \Cref{prop:finite-reduction}.

For a fixed conjugacy class \(\Cl(g)\subset \PGL(V)\), we consider the incidence locus
\[
\mathfrak I_g
:=
\{ (h,X)\in \Cl(g)\times T^\circ \mid h(X)=X \}.
\]
Its image in \(T^\circ\) is precisely the locus of smooth complete intersections
stabilized by some element of \(\Cl(g)\). Thus the desired generality statement follows
once one proves
\(
\dim \mathfrak I_g<\dim T^\circ .
\)
The rest of the argument is devoted to proving this dimension estimate for the finitely
many conjugacy classes left by the reduction.

The rest of the paper is devoted to proving the required incidence estimates. The
prime-to-characteristic case is semisimple, while the order-\(p\) cases in positive
characteristic are unipotent. In characteristic \(2\), the order-\(2\) case requires a
separate square-zero argument, so we treat it independently.

\subsubsection{Automorphisms of order prime to \texorpdfstring{$p$}{p}}

In the prime-to-characteristic case, the relevant automorphisms are semisimple by
Maschke's theorem, and the problem is reduced to estimating loci of stable subspaces in
suitable Grassmannians. We prove the following theorem.

\begin{theorem}\label{thm:ptp}
Let \(k\) be an algebraically closed field of characteristic \(p\ge 0\). Let \(X\subset
\PP^n_k\) be a general smooth complete intersection of multidegree
\[
2\le d_1\le \cdots \le d_c,\qquad c\ge 2,\qquad n-c\ge 1.
\]
Assume \((d_1,\dots,d_c)\neq (2,2)\). Then \(X\) admits no nontrivial linear
automorphism whose order is prime to \(p\), where in characteristic \(0\) this means
every finite order.
\end{theorem}

\subsubsection{Automorphisms of order \texorpdfstring{$p$}{p} in odd characteristic}

Assume now that \(\operatorname{char} k=p>2\). In the order-\(p\) case, the relevant
automorphisms are unipotent, so the semisimple decomposition is no longer available. We
instead use image-rank estimates for the induced nilpotent operator on equation spaces
and compare the resulting codimension with the dimension of the corresponding conjugacy
class. We prove the following theorem.

\begin{theorem}\label{thm:op}
Let \(k\) be an algebraically closed field of characteristic \(p>2\). Let \(X\subset
\PP^n_k\) be a general smooth complete intersection of multidegree
\[
2\le d_1\le \cdots \le d_c,\qquad c\ge 2,\qquad n-c\ge 1.
\]
Assume \((d_1,\dots,d_c)\neq (2,2)\). Then \(X\) admits no nontrivial linear
automorphism of order \(p\).
\end{theorem}

\subsubsection{Involutions in characteristic \texorpdfstring{$2$}{2}}

It remains to treat the case \(\operatorname{char} k=2\). Here an order-\(2\) linear
automorphism is unipotent and the induced nilpotent operator has square zero. This
special structure leads to different estimates, especially in the quadratic cases. We
prove the following theorem.

\begin{theorem}\label{thm:c2}
Let \(k\) be an algebraically closed field of characteristic \(2\). Let \(X\subset
\PP^n_k\) be a general smooth complete intersection of multidegree
\[
2\le d_1\le \cdots \le d_c,\qquad c\ge 2,\qquad n-c\ge 1.
\]
Assume \((d_1,\dots,d_c)\neq (2,2)\). Then \(X\) admits no nontrivial linear
automorphism of order \(2\).
\end{theorem}
\section{Prime-to-characteristic automorphisms}\label{sec:ptp}

Throughout this section, \(k\) is an algebraically closed field of characteristic \(p\ge
0\). We consider the action of order prime to \(p\); in
characteristic \(0\), this means every positive integer.

This section is devoted to the proof of \Cref{thm:ptp}. The proof is organized according
to the degrees of the defining equations. Write the multidegree as \(2\le d_1\le \cdots
\le d_c\), and let \(d_1=\cdots=d_r=a\le d_c\) with \(r\) maximal.

The first case is the equal-degree initial block case. This means that the lowest-degree
equations form a block of type \((a^r)\). If this block is already non-exceptional,
namely if either \(a\ge 3\) and \(r\ge 2\), or \(a=2\) and \(r\ge 3\), then the
lowest-degree equations alone are sufficient for the automorphism estimate. This case is
treated in \Cref{sec:eq}.

The second case is the mixed-degree case. This is the case where the lowest-degree block
alone is not sufficient for the argument. This happens either when the lowest degree
occurs only once, or when the lowest block is the exceptional two-quadric block
\((2,2)\). In these cases, the next distinct degree has to be used. If the next degree
is \(b>a\) and occurs \(s\) times, then the new degree-\(b\) equations are studied
through their minimal equation space modulo the ideal generated by the lower-degree
equations. The mixed initial patterns needed later are therefore
\[
(2,2,b^s),\qquad (2,b^s),\qquad (a,b^s)\quad (3\le a<b).
\]
The reduction to the lowest-degree part and the definition of these minimal equation
spaces are recorded in \Cref{sec:fl}, and the mixed-degree estimates are carried out in
\Cref{sec:mix}.

The common dimension estimate for both cases is the Grassmannian incidence criterion
proved in \Cref{sec:inc}. Before those degree estimates, the following subsection gives
the finiteness reduction and the basic incidence setup. These ingredients are combined
at the end of the section in the proof of \Cref{thm:ptp}. Before entering the
degree-by-degree analysis, we first explain why it is enough to consider only finitely
many prime-order conjugacy classes; this is the content of \Cref{sec:fin}.

\subsection{Finiteness reduction and incidence setup}\label{sec:fin}

The first step is a finiteness reduction. Let $T^\circ$ be the smooth parameter space
of complete intersections of the fixed multidegree, and set
\[
\mathcal A
=
\{(X,h)\in T^\circ\times \mathrm{PGL}(V)\mid h(X)=X\}.
\]

\begin{proposition}\label{prop:finite-reduction}
The projection
\(
\mathcal{A}\longrightarrow T^\circ
\)
is finite. Consequently, there exists an integer \(M\) such that
\begin{equation}\label{eq:autbd}
|\operatorname{Aut}_L(X)|\le M
\end{equation}
for every \(X\in T^\circ\). In particular, to exclude nontrivial linear
automorphisms of order prime to \(p\), it is enough to exclude elements lying
in finitely many prime-order conjugacy classes in \(\PGL(V)\).
\end{proposition}

\begin{proof}
The finiteness of
\(
\mathcal{A}\longrightarrow T^\circ
\)
follows from Benoist's separatedness theorem
\cite[Lemma~2.2 and Théorème~1.7]{Benoist2013}.
Since \(T^\circ\) is noetherian, the lengths of the fibers of this finite
morphism are uniformly bounded. Hence there exists an integer \(M\) such that
\(
|\operatorname{Aut}_L(X)|\le M
\)
for every \(X\in T^\circ\).

Suppose that \(\operatorname{Aut}_L(X)\) contains a nontrivial element of
order \(m\) prime to \(p\). If \(\ell\) is a prime divisor of \(m\), then
\(\operatorname{Aut}_L(X)\) contains an element of order \(\ell\). Necessarily,
\(
\ell\le |\operatorname{Aut}_L(X)|\le M,
\)
so only finitely many primes \(\ell\ne p\) can occur.

For each fixed \(\ell\ne p\), an element of order \(\ell\) in \(\PGL(V)\)
may be lifted, after rescaling, to an element \(g\in\GL(V)\) satisfying
\(g^\ell=I\). Since \(\ell\ne p\), such an element is semisimple. Its
conjugacy class is determined, up to a common scalar, by the multiplicities
of its eigenvalues in \(\mu_\ell\). Hence there are only finitely many
conjugacy classes of elements of order \(\ell\) in \(\PGL(V)\).

Thus only finitely many prime-order conjugacy classes have to be excluded.
\end{proof}

For each of these conjugacy classes, the same incidence idea as in Lyu--Zhang
\cite[Proposition~2.15, pp.~16--19]{LyuZhang2024} is used. Namely, fix an element
\(g\in \mathrm{PGL}(V)\) of order \(\ell\), and let \(\mathrm{Cl}(g)\) be its conjugacy
class. Consider the incidence variety
\[
\mathfrak I_g
=
\{(h,X)\in \mathrm{Cl}(g)\times T^\circ\mid h(X)=X\}.
\]
Let \(Z_g:=\operatorname{pr}_2(\mathfrak I_g)\subset T^\circ\) be the locus of smooth
complete intersections stabilized by some element in \(\mathrm{Cl}(g)\).

\begin{proposition}\label{prop:fixed-incidence}
With notation as above, suppose that
\(
\dim Z_g<\dim T^\circ,
\)
where the dimension of the constructible set \(Z_g\) is understood as the dimension of
its Zariski closure. Then \(Z_g\) is contained in a proper closed subset of \(T^\circ\).
In particular, the same conclusion holds if
\(
\dim \mathfrak I_g<\dim T^\circ,
\)
since \(\dim Z_g\le \dim \mathfrak I_g\). Hence a general member of \(T^\circ\) is not
stabilized by any element in \(\mathrm{Cl}(g)\). If the corresponding estimate holds for
all conjugacy classes in a finite collection, then the union of the corresponding
stabilized loci is still contained in a proper closed subset of \(T^\circ\).
\end{proposition}

\begin{proof}
By Chevalley's theorem, \(Z_g\) is constructible. By assumption, the dimension of its
Zariski closure is strictly smaller than \(\dim T^\circ\). Hence the Zariski closure
\(\overline{Z_g}\) is a proper closed subset of \(T^\circ\), and therefore \(Z_g\) is
contained in a proper closed subset of \(T^\circ\). Thus a general member of
\(T^\circ\) is not stabilized by any element in \(\Cl(g)\).

If instead one has the stronger estimate
\(
\dim \mathfrak I_g<\dim T^\circ,
\)
then the same conclusion follows from \(\dim Z_g\le \dim \mathfrak I_g\).

Finally, if the estimate holds for all conjugacy classes in a finite collection, then
the union of the corresponding stabilized loci is contained in the finite union of
proper closed subsets of \(T^\circ\), which is again a proper closed subset of
\(T^\circ\).
\end{proof}

In the arguments below, \(T^\circ\) will usually be an open subset of a Grassmannian.
For instance, in the equal-degree case of type \((d^r)\), if
\(S_d=\operatorname{Sym}^d(V^\vee)\), the defining equations span an \(r\)-dimensional
subspace of \(S_d\), and \(T^\circ\) is the smooth open subset of
\(\operatorname{Gr}(r,S_d)\). Therefore, in that case,
\begin{equation}\label{eq:grdim}
\dim T^\circ=\dim \operatorname{Gr}(r,S_d)=r(\dim S_d-r).
\end{equation}
This is the dimension against which the dimension of the corresponding stabilized locus
will be compared below.

\begin{remark}\label{rem:pgl}
Let \(g\in\PGL(V)\) have prime order \(\ell\neq p\). After choosing and rescaling a lift, we may rewrite \(g\) as an element of \(\GL(V)\) satisfying $g^\ell=I$.

Indeed, any lift satisfies \(g^\ell=\lambda I\) for some \(\lambda\in k^\times\), and such a normalization is possible because \(k\) is algebraically closed. The normalized lift is semisimple since \(\ell\) is invertible in \(k\).

Replacing a lift by a scalar multiple only multiplies its action on \(\Sym^d(V^\vee)\) by a scalar, and hence does not change its invariant subspaces, nor those in any equivariant subquotient of fixed degree. Moreover, the natural map
\(
\Cl_{\GL(V)}(g)\longrightarrow\Cl_{\PGL(V)}(g)
\)
is surjective with finite fibers. Indeed, two elements in the same fiber differ by a scalar \(c\), and if both are conjugate to the normalized lift \(g\), then \(c^\ell=1\). Consequently,
\(
\dim\Cl_{\GL(V)}(g)=\dim\Cl_{\PGL(V)}(g).
\)
Thus we henceforth use \(g\) for a normalized lift and compute conjugacy-class dimensions in \(\GL(V)\).
\end{remark}

\subsection{Reduction to the lowest-degree part and minimal equation spaces}\label{sec:fl}
This subsection explains that, for the purpose of estimating linear automorphisms, it
is not always necessary to use all defining equations simultaneously. The equations of
minimal degrees are automatically preserved by every linear automorphism.

We first record the basic reduction to the lowest-degree equations. The reductions in
this subsection are geometric and do not use semisimplicity, and they are also used later in the order-\(p\)
arguments in the next section.

\begin{lemma}\label{lem:fl}
Let $X\subset \PP^n_k$ be a complete intersection of multidegree $(d_1,\dots,d_c)$, and
suppose \(d_1=\cdots=d_r<d_{r+1}\le \cdots \le d_c\) for some maximal $r$. Choose
defining equations $F_1,\dots,F_c$ with $\deg F_i=d_i$, and set
\(Y:=V(F_1,\dots,F_r)\subset \PP^n\). Then every linear automorphism of $X$ preserves
the ideal $(F_1,\dots,F_r)$, hence \(\AutL(X)\subset \AutL(Y)\).
\end{lemma}

\begin{proof}
The forms $F_1,\dots,F_r$ span the smallest nonzero graded piece of $I_X$. A linear
automorphism of $X$ preserves $I_X$ degree by degree, hence preserves this smallest
nonzero graded piece and therefore the ideal it generates. Thus it induces a linear
automorphism of \(Y=V(F_1,\dots,F_r)\).
\end{proof}

In some cases the lowest-degree equations alone are not enough for the dimension
estimate, and one has to use the next minimal equation space as well. We call these the
mixed-degree cases. Here are the technical inputs for them. Let \(e_1<e_2<\cdots<e_m\)
be the distinct degrees occurring among $d_1,\dots,d_c$. For each $j$ define
\[
J_{<e_j}:=\text{the ideal generated by }\bigoplus_{t<e_j}(I_X)_t.
\]
We then define the minimal equation space in degree $e_j$ by
\begin{equation}\label{eq:minsp}
N_{e_j}(X):=(I_X)_{e_j}/(J_{<e_j})_{e_j}.
\end{equation}

\begin{lemma}\label{lem:new}
Every linear automorphism of \(X\) induces a linear automorphism of each minimal
equation space \(N_{e_j}(X)\).
\end{lemma}

\begin{proof}
A linear automorphism of \(X\) preserves \(I_X\) degree by degree. Since it preserves
all lower-degree pieces, it preserves the ideal generated by them, and hence preserves
\((J_{<e_j})_{e_j}\). Therefore it induces a linear automorphism of the quotient
\(
N_{e_j}(X)=(I_X)_{e_j}/(J_{<e_j})_{e_j}.
\)
\end{proof}

\begin{proposition}\label{prop:bt}
Let \(Y\subset \PP^n_k\) be a fixed smooth complete intersection with homogeneous ideal
\(I_Y\subset S=\Sym(V^\vee)\). Let \(e\ge 1\), and assume that \(I_Y\) is generated by
forms of degrees strictly smaller than \(e\). Set
\(
R_e(Y):=H^0(Y,\mathcal O_Y(e))\simeq S_e/(I_Y)_e.
\)
Let \(r\ge 1\) and assume that \(\dim Y\ge r\). Then there exists a nonempty Zariski open
subset
\(
V_{Y,e,r}\subset \Gr(r,R_e(Y))
\)
such that, for every \(U\in V_{Y,e,r}\), the common zero locus of the sections in \(U\)
is a smooth complete intersection of codimension \(r\) in \(Y\).

Moreover, if \(\widetilde U\subset S_e\) is any \(r\)-dimensional subspace mapping
isomorphically onto \(U\) under the quotient map
\(
S_e\twoheadrightarrow R_e(Y)=S_e/(I_Y)_e,
\)
then
\(
X:=Y\cap V(\widetilde U)\subset \PP^n_k
\)
is a smooth complete intersection obtained from \(Y\) by adding \(r\) equations of
degree \(e\), and its degree-\(e\) minimal equation space is naturally identified
with \(U\).
\end{proposition}

\begin{proof}
Let
\(
G:=\Gr(r,R_e(Y)).
\)
The universal rank-\(r\) subbundle on \(G\) determines, on \(G\times Y\), a universal
family of \(r\) sections of \(\mathcal O_Y(e)\). The locus of points \(U\in G\) for
which these sections cut out a smooth subscheme of codimension \(r\) in \(Y\) is
Zariski open.

We show that this open subset is nonempty. Since \(\mathcal O_Y(e)\) is very ample,
Bertini's theorem gives a section
\(
s_1\in H^0(Y,\mathcal O_Y(e))
\)
whose zero locus \(Y_1:=V_Y(s_1)\) is smooth of codimension \(1\) in \(Y\). The
restriction of \(\mathcal O_Y(e)\) to \(Y_1\) is again very ample. Applying Bertini
successively, we may choose linearly independent sections
\(
s_1,\dots,s_r\in H^0(Y,\mathcal O_Y(e))
\)
such that
\(
Y_j:=V_Y(s_1,\dots,s_j)
\)
is smooth of codimension \(j\) in \(Y\) for every \(1\le j\le r\). Hence
\(
U:=\langle s_1,\dots,s_r\rangle
\)
belongs to the required open subset, which we denote by \(V_{Y,e,r}\).

Now let \(U\in V_{Y,e,r}\), and let \(\widetilde U\subset S_e\) be any lift of \(U\).
The restrictions to \(Y\) of the forms in \(\widetilde U\) are precisely the sections
in \(U\). Therefore
\(
X=Y\cap V(\widetilde U)
\)
is smooth of codimension \(r\) in \(Y\). Since \(Y\) is a complete intersection in
\(\PP^n_k\), adjoining a basis of \(\widetilde U\) to a regular sequence defining
\(Y\) gives a regular sequence defining \(X\). Thus \(X\) is a smooth complete
intersection in \(\PP^n_k\).

Finally,
\(
I_X=I_Y+(\widetilde U).
\)
Since \(I_Y\) is generated in degrees strictly smaller than \(e\), while the new
generators have degree \(e\), the ideal generated by the components of \(I_X\) of
degree less than \(e\) is precisely \(I_Y\). Hence
\(
J_{<e}=I_Y
\)
and therefore
\[
N_e(X)
=
(I_X)_e/(J_{<e})_e
=
\bigl((I_Y)_e+\widetilde U\bigr)/(I_Y)_e
\simeq U.
\]
This identification is independent of the choice of \(\widetilde U\).
\end{proof}
\begin{remark}\label{rem:mixed-reduction}
In the mixed-degree cases
\[
(2,b^s),\qquad (a,b^s)\ (3\le a<b),\qquad (2,2,b^s),
\]
we may restrict to the nonempty open locus where the lower-degree part \(Y\)
is smooth. Since our statements concern a general complete intersection, this
restriction is harmless.

By \cref{lem:fl,lem:new}, every linear automorphism preserves \(Y\) and acts on
the minimal equation space in the next degree. After fixing \(Y\),
\cref{prop:bt} shows that the minimal equation spaces giving smooth complete
intersections contain a nonempty open subset of the relevant Grassmannian.
Thus, in the later incidence arguments, it is enough to consider the equations
in the first one or two distinct degrees.
\end{remark}
The following lemma will be used for the mixed case beginning with \((2,2)\):
it shows that a general higher-degree minimal equation space is stabilized by
none of the finitely many non-scalar automorphisms of the fixed lower-degree
part.

\begin{lemma}\label{lem:kill}
Let \(M\) be a finite-dimensional \(k\)-vector space, and let \(1\le s<\dim M\). Let
\(T\subset \operatorname{End}_k(M)\setminus k\cdot\id_M\) be a finite subset. Then a
general point \(U\in\Gr(s,M)\) satisfies \(h(U)\nsubseteq U\) for every \(h\in T\). In
particular, if \(T\subset\GL(M)\), then a general \(s\)-plane is stabilized by no element
of \(T\).
\end{lemma}

\begin{proof}
For \(h\in T\), let
\(
Z_h:=\{U\in\Gr(s,M)\mid h(U)\subset U\}.
\)
This is closed in \(\Gr(s,M)\). Since \(h\) is not scalar, choose \(v\in M\) with
\(h(v)\notin kv\). As \(s<\dim M\), there exists an \(s\)-plane \(U\subset M\) containing
\(v\) but not \(h(v)\). Then \(h(U)\nsubseteq U\), so \(Z_h\) is proper. Hence
\(\bigcup_{h\in T}Z_h\) is a proper closed subset, and its complement is the required
general locus. If \(h\) is invertible, the condition \(h(U)\subset U\) is equivalent to
\(h(U)=U\).
\end{proof}

\subsection{A Grassmannian incidence criterion}\label{sec:inc}

Now we start our main calculation. Let $G$ be an algebraic group acting linearly on a
finite-dimensional $k$-vector space $M$, and fix an integer $r\ge 1$. Let $g\in G$ be a
semisimple element. Since $g$ acts semisimply on $M$, we have a decomposition
\(M=\bigoplus_\lambda M_\lambda\), where \(\lambda\) runs over the eigenvalues of \(g\)
on \(M\), and \(M_\lambda:=\{v\in M: g\cdot v=\lambda v\}\) is the eigensubspace. Set
\[
\omega(M):=\max_\lambda \dim M_\lambda,
\qquad
\gamma(M):=\dim M-\omega(M).
\]
We recall the following incidence notation. Let
\[
\mathfrak I_g(r,M)
:=
\{(h,U)\in \Cl_G(g)\times \Gr(r,M)\mid h(U)=U\},
\]
where \(\Cl_G(g)\) denotes the conjugacy class of \(g\) in \(G\). Let
\[
Z_g(r,M):=\operatorname{pr}_2\bigl(\mathfrak I_g(r,M)\bigr)\subset \Gr(r,M)
\]
be the locus of \(r\)-planes stabilized by some element in \(\Cl_G(g)\). The condition \(h(U)=U\) is closed in
\(\Cl_G(g)\times \Gr(r,M)\). As in
\Cref{prop:fixed-incidence}, \(Z_g(r,M)\) is constructible, and after interpreting its
dimension as that of its Zariski closure, we have
\(
\dim Z_g(r,M)\le \dim \mathfrak I_g(r,M).
\)
Thus it suffices to bound the dimension of \(\mathfrak I_g(r,M)\).

\begin{proposition}\label{prop:inc}
With notation as above, we have
\begin{equation}\label{eq:incbd}
\dim Z_g(r,M)
\le
\dim \Cl_G(g)+r\omega(M)-r.
\end{equation}
In particular, a sufficient condition for
\[
\dim Z_g(r,M)
<
\dim \Gr(r,M)
=
r(\dim M-r)
\]
is
\begin{equation}\label{eq:inccrit}
r\gamma(M)
>
\dim \Cl_G(g)+r(r-1).
\end{equation}
\end{proposition}

\begin{proof}
Fix \(h\in \Cl_G(g)\), and write
\(
M=\bigoplus_\lambda M_\lambda
\)
for its eigenspace decomposition. Since \(h\) is semisimple, every
\(h\)-stable \(r\)-plane \(U\subset M\) decomposes as
\[
U=\bigoplus_\lambda U_\lambda,
\qquad
U_\lambda\subset M_\lambda.
\]
If \(f_\lambda:=\dim U_\lambda\), then \(\sum_\lambda f_\lambda=r\), and
for a fixed admissible \((f_\lambda)\), the corresponding locus of
\(h\)-stable \(r\)-planes is
\(
\prod_\lambda \Gr(f_\lambda,M_\lambda).
\)
Its dimension is therefore
$
\sum_\lambda
f_\lambda\bigl(\dim M_\lambda-f_\lambda\bigr).
$

All elements of \(\Cl_G(g)\) have the same eigenspace multiplicities on
\(M\). Hence
\[
\dim Z_g(r,M)
\le
\dim \Cl_G(g)
+
\max_{\sum_\lambda f_\lambda=r}
\sum_\lambda
f_\lambda\bigl(\dim M_\lambda-f_\lambda\bigr),
\]
where the maximum is taken over all the admissible tuples.

Since \(\dim M_\lambda\le \omega(M)\) for every \(\lambda\), we have
\[
\begin{aligned}
\sum_\lambda
f_\lambda\bigl(\dim M_\lambda-f_\lambda\bigr)
&\le
r\omega(M)-\sum_\lambda f_\lambda^2 
\le
r\omega(M)-r,
\end{aligned}
\]
because the \(f_\lambda\) are nonnegative integers and
\(
\sum_\lambda f_\lambda^2
\ge
\sum_\lambda f_\lambda
=
r.
\)
This proves \eqref{eq:incbd}.

Finally,
\[
\begin{gathered}
\dim \Cl_G(g)+r\omega(M)-r
<
r(\dim M-r) \\
\Longleftrightarrow \\
\dim \Cl_G(g)+r\omega(M)-r
<
r\dim M-r^2 \\
\Longleftrightarrow \\
\dim \Cl_G(g)+r(r-1)
<
r\bigl(\dim M-\omega(M)\bigr) \\
\Longleftrightarrow \\
\dim \Cl_G(g)+r(r-1)
<
r\gamma(M).
\end{gathered}
\]
Thus \eqref{eq:inccrit} is a sufficient condition.
\end{proof}

\subsection{The equal-degree equations case}\label{sec:eq}

In this section we study general complete intersections whose lowest-degree block
consists of several equations of the same degree, say $(d^r)$. More specifically, in
this section we will prove:
\begin{theorem}\label{thm:eq}
Let $k$ be an algebraically closed field of characteristic $p\ge 0$. Let \(X\subset
\PP^n_k\) be a general complete intersection of equal multidegree $(d^r)$, with \(n-r\ge
1\). Assume either \(d\ge 3,\ r\ge 2\), or \(d=2,\ r\ge 3\). Then $X$ admits no
nontrivial linear automorphism of prime order $\ell\neq p$.
\end{theorem}

We again introduce some notations. Let $V$ be an $(n+1)$-dimensional vector space, and
set
\[
S:=\Sym(V^\vee),\qquad S_q:=\Sym^q(V^\vee),\qquad E_q:=\dim S_q=\binom{n+q}{q}.
\]
Whenever a semisimple element \(g\) acts on \(S_q\), write \(S_q=\bigoplus_t (S_q)_t\)
for its eigenspace decomposition, and set
\[
\omega_q:=\max_t \dim (S_q)_t,
\qquad
\gamma_q:=E_q-\omega_q.
\]
More generally, if \(A=\bigoplus_{e\ge 0}A_e\) is a graded semisimple \(g\)-algebra,
write
\(
\gamma_e(A):=\dim A_e-\max_t\dim (A_e)_t.
\)
\begin{lemma}\label{lem:qgapmon}
Let \(A=\bigoplus_{e\ge 0}A_e\) be a graded domain with a semisimple action of \(g\)
preserving the grading. Assume \(A_1\neq 0\). Then
\begin{equation}\label{eq:qgapmon}
\gamma_e(A)\le \gamma_{e+1}(A)
\end{equation}
for every \(e\ge 0\).
\end{lemma}

\begin{proof}
Choose a nonzero \(g\)-eigenvector \(z\in A_1\), say with eigenvalue \(\alpha\). Since
\(A\) is a domain, multiplication by \(z\) gives an injective map \(A_e\hookrightarrow
A_{e+1}\). This map sends the \(\lambda\)-eigenspace \((A_e)_\lambda\) into the
\(\alpha\lambda\)-eigenspace \((A_{e+1})_{\alpha\lambda}\).

Let \(\mu\) be an eigenvalue for which \((A_{e+1})_\mu\) has maximal dimension. Then the
image of
\[
\bigoplus_{\lambda\neq \alpha^{-1}\mu}(A_e)_\lambda
\]
is contained in
\[
\bigoplus_{\nu\neq \mu}(A_{e+1})_\nu.
\]
Hence
\[
\gamma_{e+1}(A)
=
\sum_{\nu\neq \mu}\dim(A_{e+1})_\nu
\ge
\sum_{\lambda\neq \alpha^{-1}\mu}\dim(A_e)_\lambda
\ge
\gamma_e(A).
\]
\end{proof}
\subsubsection{Prime-order elements}

By \Cref{rem:pgl}, when a projective automorphism has prime order \(\ell\neq p\), we may
represent it by a normalized lift \(g\in\GL(V)\) satisfying \(g^\ell=I\), and the
conjugacy-class dimension may be computed in \(\GL(V)\). Thus let \(g\in \GL(V)\) have
prime order \(\ell\neq p\), and assume its image in \(\PGL(V)\) is nontrivial. Then \(g\)
is semisimple. Fix a primitive \(\ell\)-th root of unity
\(\zeta\in k\). We index eigenspaces by their weights in \(\ZZ/\ell\ZZ\): for \(a\in
\ZZ/\ell\ZZ\), the subspace \(V_a\) is the \(\zeta^a\)-eigenspace of \(g\) on \(V\).
Thus, if
\[
V=\bigoplus_{a\in \ZZ/\ell\ZZ} V_a,
\qquad
n_a:=\dim V_a,
\]
then \(\sum_a n_a=n+1\). The induced action of \(g\) on each \(S_q=\Sym^q(V^\vee)\) is
likewise semisimple. We keep the notation
\[
S_q=\bigoplus_{t\in \ZZ/\ell\ZZ}(S_q)_t,
\qquad
\omega_q=\max_t\dim(S_q)_t,
\qquad
\gamma_q=E_q-\omega_q.
\]
We use the formula
\begin{equation}\label{eq:Dg}
\begin{aligned}
D_g
&:=\dim \Cl_{\GL(V)}(g) \\
&=\dim \GL(V)-\dim Z_{\GL(V)}(g) \\
&=(n+1)^2-\sum_a n_a^2,
\end{aligned}
\end{equation}
where \(Z_{\GL(V)}(g)\) denotes the centralizer of \(g\). We shall also use
\begin{equation}\label{eq:Dglb}
D_g\ge 2n
\end{equation}
for every non-scalar semisimple element. Indeed, the minimum occurs for eigenspace
multiplicities $(n,1)$.

\subsubsection{The case \texorpdfstring{$d\ge 3$}{d>=3}}

\begin{lemma}\label{lem:d3gap}
Assume \(d\ge 3\) and \(\dim V=n+1\ge 4\). Then
\begin{equation}\label{eq:d3gap}
\gamma_d>\frac12D_g+1.
\end{equation}
Moreover, \(D_g\) is even and \(\gamma_d\) is an integer, and so \(\gamma_d\ge
\frac12D_g+2\). \end{lemma}

\begin{proof}
Put \(m:=\dim V=n+1\), \(n_{\max}:=\max_a n_a\), and
\(\rho:=m-n_{\max}\ge 1\). When passing from \(V\) to \(V^\vee\), the weights are
negated. This only relabels the eigenspaces, so the eigenspace multiplicities,
\(n_{\max}\), and the orbit dimension \(D_g=m^2-\sum_a n_a^2\) are unchanged. Moreover,
\begin{equation}\label{eq:Dg-even}
D_g=m^2-\sum_a n_a^2\equiv m-\sum_a n_a=0\pmod 2.
\end{equation}
Thus \(D_g\) is even.

We first prove the monomial estimate in a slightly more explicit way. Choose an
eigenbasis \(x_0,\dots,x_{m-1}\) of \(V^\vee\). From now on, the weights of variables
and monomials are the weights for the induced action on \(V^\vee\). By the preceding
observation, at most \(n_{\max}\) variables have any prescribed weight.

Fix \(t\in \ZZ/\ell\ZZ\). We count the degree-\(d\) monomials of weight \(t\) with one
marked occurrence. Namely, write a monomial \(Q=x_{i_1}\cdots x_{i_d}\), with
\(i_1\le\cdots\le i_d\), as an ordered list of its \(d\) variable-occurrences, and form
the \(d\) marked monomials obtained by putting a dagger on the first, second,
\(\ldots\), \(d\)-th occurrence:
\[
x_{i_1}^{\dagger}x_{i_2}\cdots x_{i_d},\quad
x_{i_1}x_{i_2}^{\dagger}\cdots x_{i_d},\quad \ldots,\quad
x_{i_1}\cdots x_{i_d}^{\dagger}.
\]
If several consecutive occurrences are the same variable, we still count the different
choices of the marked occurrence separately. Thus the total number of such marked
monomials is \(d\cdot \dim(S_d)_t\).

Now erase the daggered variable. What remains is a monomial \(P\) of degree \(d-1\),
and there are \(\binom{m+d-2}{d-1}\) possibilities for \(P\). We claim that, for each
fixed \(P\), at most \(n_{\max}+d-1\) marked monomials can give rise to it.

Indeed, write \(P=x_0^{\alpha_0(P)}\cdots x_{m-1}^{\alpha_{m-1}(P)}\). If the erased
daggered variable is \(x_j\), then the original monomial was \(x_jP\), so
\(\operatorname{wt}(x_j)=t-\operatorname{wt}(P)\). There are at most \(n_{\max}\)
variables of this prescribed weight. For such a variable \(x_j\), the monomial \(x_jP\)
contains \(\alpha_j(P)+1\) occurrences of \(x_j\), and the dagger may be placed on any
one of these occurrences. Hence the number of marked monomials lying over this fixed
\(P\) is
\[
\sum_{\operatorname{wt}(x_j)=t-\operatorname{wt}(P)}
\bigl(\alpha_j(P)+1\bigr)
\le
n_{\max}
+
\sum_{\operatorname{wt}(x_j)=t-\operatorname{wt}(P)}
\alpha_j(P)
\le
n_{\max}+d-1.
\]
Therefore
\[
d\cdot \dim(S_d)_t\le (n_{\max}+d-1)\binom{m+d-2}{d-1}.
\]
Since
\(E_d=\frac{m+d-1}{d}\binom{m+d-2}{d-1}\), we have, for every weight \(t\),
\[
\begin{aligned}
E_d-\dim(S_d)_t
&\ge
\left(\frac{m+d-1}{d}-\frac{n_{\max}+d-1}{d}\right)
\binom{m+d-2}{d-1} \\
&=
\frac{m-n_{\max}}{d}\binom{m+d-2}{d-1}
=
\frac{\rho}{d}\binom{m+d-2}{d-1}.
\end{aligned}
\]
Taking the minimum over the complements of all weight spaces gives
\begin{equation}\label{eq:d3mon}
\gamma_d
\ge
\frac{\rho}{d}\binom{m+d-2}{d-1}.
\end{equation}

For fixed \(m\), set \(A_d:=\frac1d\binom{m+d-2}{d-1}\). Then, for \(d\ge3\),
\[
\frac{A_{d+1}}{A_d}
=
\frac{d}{d+1}\cdot\frac{m+d-1}{d}
=
\frac{m+d-1}{d+1}
=
1+\frac{m-2}{d+1}
>1,
\]
because \(m\ge4\). Hence \(A_d\ge A_3\) for all \(d\ge3\). Since
\(A_3=\frac13\binom{m+1}{2}=\frac{m(m+1)}6\), \eqref{eq:d3mon} implies, for all
\(d\ge3\),
\begin{equation}\label{eq:d3rho}
\gamma_d\ge \frac{\rho m(m+1)}6.
\end{equation}

We also use the elementary upper bound for the orbit dimension. If \(a_0\) is chosen so
that \(n_{a_0}=m-\rho\), then the remaining eigenspaces have total dimension \(\rho\).
Therefore \(\sum_{a\ne a_0}n_a^2\ge \rho\), and hence
\begin{equation}\label{eq:d3D}
D_g
=m^2-\sum_a n_a^2
\le
m^2-(m-\rho)^2-\rho
=\rho(2m-\rho-1).
\end{equation}

Assume first that \(\rho\ge2\). Combining \eqref{eq:d3rho} and \eqref{eq:d3D}, we obtain
\[
\begin{aligned}
\gamma_d-\left(\frac12D_g+1\right)
&\ge
\frac{\rho m(m+1)}6
-
\frac{\rho(2m-\rho-1)}2
-1 \\
&=
\frac{\rho(m^2-5m+3\rho+3)}6-1.
\end{aligned}
\]
For fixed \(m\ge4\), the function
\(\rho\mapsto \rho(m^2-5m+3\rho+3)\) is increasing on \([2,\infty)\), since its
derivative is \(m^2-5m+6\rho+3\ge 16-20+12+3>0\). Hence, for \(\rho\ge2\), the
right-hand side is minimized at \(\rho=2\). Thus
\[
\gamma_d-\left(\frac12D_g+1\right)
\ge
\frac{2(m^2-5m+9)}6-1
=
\frac{(m-2)(m-3)}3>0.
\]

It remains to treat \(\rho=1\). Then the eigenspace multiplicities are \((m-1,1)\), and
\(D_g=2m-2\). We continue to work with the induced weights on \(V^\vee\). Let \(G\in\GL(V)\) be the
chosen lift, and suppose that the \((m-1)\)-dimensional eigenspace in \(V^\vee\) has
weight \(\mu\); equivalently, \(G^\vee\) acts on this eigenspace by the scalar
\(\zeta^\mu\), where \(\zeta\) is a primitive \(\ell\)-th root of unity. Replacing
\(G\) by the scalar multiple \(G':=\zeta^\mu G\), as allowed by \Cref{rem:pgl}, does
not change the associated element of \(\PGL(V)\), and still gives a normalized lift.
Moreover, $(G')^\vee=(\zeta^\mu G)^\vee=\zeta^{-\mu}G^\vee$. Thus, if \(x\in (V^\vee)_a\), so that \(G^\vee x=\zeta^a x\), then $(G')^\vee x=\zeta^{a-\mu}x$. In particular, the \((m-1)\)-dimensional eigenspace, for which \(a=\mu\), has weight
\(0\) for \(G'\). On \(S_d=\Sym^d(V^\vee)\), this only translates the degree-\(d\)
weights by \(-d\mu\), and hence merely relabels the weight spaces. Therefore
\(D_g\) and \(\gamma_d\) are unchanged, and we may assume that the
\((m-1)\)-dimensional eigenspace in \(V^\vee\) has weight \(0\). Let \(A\subset V^\vee\) be this eigenspace, and let
\(y\) generate the remaining one-dimensional eigenspace, of weight \(\alpha\neq0\).

For the weight \(0\), the complement contains all monomials of type \(A^{d-1}y\), and
this space has dimension \(\binom{m+d-3}{d-1}\). For the weight \(\alpha\), the
complement contains all monomials of type \(A^d\), whose dimension is
\(\binom{m+d-2}{d}\ge \binom{m+d-3}{d-1}\) because \((m+d-2)/d\ge1\). For any other
weight, the complement contains both of these two families. Therefore the complement of
every weight space has dimension at least \(\binom{m+d-3}{d-1}\), and hence
\(\gamma_d\ge \binom{m+d-3}{d-1}\).

Since \(d\ge3\), writing \(r=d-1\ge2\), we have
\(
\binom{m+d-3}{d-1}
=
\binom{m+r-2}{r}
\ge
\binom m2;
\)
indeed the sequence \(\binom{m+r-2}{r}\) is increasing for \(r\ge2\), and its value at
\(r=2\) is \(\binom m2\). Therefore
\(\gamma_d\ge\binom m2>m=\frac12D_g+1\).

This proves \eqref{eq:d3gap}. Since \(\gamma_d\in\ZZ\) and \(D_g\) is even by
\eqref{eq:Dg-even}, the strict inequality \eqref{eq:d3gap} is equivalent to
\(\gamma_d\ge \frac12D_g+2\). This completes the proof.
\end{proof}

\begin{proposition}\label{prop:d3}
Let \(X\subset \PP^n_k\) be a general complete intersection of type $(d^r)$ with \(d\ge
3, \\r\ge 2,n-r\ge 1\). Then $X$ admits no nontrivial linear automorphism of
prime order $\ell\neq p$.
\end{proposition}

\begin{proof}
By \Cref{prop:finite-reduction}, it is enough to consider the finitely many relevant
prime-order conjugacy classes. Fix one such class, represented by a nontrivial
\(g\in \PGL(V)\) of prime order \(\ell\neq p\). By \Cref{rem:pgl}, we replace \(g\) by a
normalized lift to \(\GL(V)\) and compute the conjugacy-class dimension there. For this
fixed class, \Cref{prop:fixed-incidence} reduces the generality statement to the
dimension estimate supplied by \Cref{prop:inc}. More precisely, by \eqref{eq:inccrit}, it
suffices to prove
\begin{equation}\label{eq:d3tar}
r\gamma_d>D_g+r(r-1).
\end{equation}
By \cref{lem:d3gap}, more precisely \eqref{eq:d3gap}, and by the integrality noted
there, we have \(\gamma_d\ge \frac12D_g+2\). Thus it is enough to show
\[
r\left(\frac12D_g+2\right)>D_g+r(r-1),
\]
or equivalently \(\left(\frac r2-1\right)D_g>r(r-3)\). If \(r=2\), the left-hand side is
\(0\) and the right-hand side is \(-2\), so the desired inequality holds.

Now assume \(r\ge 3\). Since \(n-r\ge 1\), we have \(n\ge r+1\). By \eqref{eq:Dglb},
\(D_g\ge 2n\ge 2r+2\). Therefore
\[
\left(\frac r2-1\right)D_g
\ge
\frac{r-2}{2}(2r+2)
=
(r-2)(r+1).
\]
For \(r=3\) this is already positive, while for \(r\ge4\), \((r-2)(r+1)-r(r-3)=2r-2>0\).
This proves \eqref{eq:d3tar}. Hence, by \Cref{prop:fixed-incidence}, the locus
stabilized by this fixed conjugacy class is contained in a proper closed subset. Taking
the finite union over the prime-order conjugacy classes allowed by
\Cref{prop:finite-reduction}, and using the finite-union assertion in
\Cref{prop:fixed-incidence}, we conclude that a general complete intersection of type
\((d^r)\) is not stabilized by any nontrivial element of prime order prime to \(p\).
\end{proof}

\subsubsection{The quadratic case: odd primes}

We now assume $d=2$ and $r\ge 3$. Let
\[
E_2=\dim \Sym^2(V^\vee)=\binom{n+2}{2}=\frac{(n+1)^2+(n+1)}{2}.
\]

\begin{lemma}\label{lem:opq1}
Assume that \(\ell\) is an odd prime distinct from \(p\). Set
$
m:=\dim V=n+1$,
$
n_{\max}:=\max_a n_a.
$
Then
\begin{equation}\label{eq:opqgap}
\gamma_2\ge \frac12D_g+\frac12(m-n_{\max}).
\end{equation}
\end{lemma}

\begin{proof}
The weights on \(V^\vee\) are the negatives of the weights on \(V\). Since replacing
\(a\) by \(-a\) only permutes the indices in \(\ZZ/\ell\ZZ\), the multiplicities are
still the same. Thus, for the purpose of counting monomials in \(S_2=\Sym^2(V^\vee)\),
we may write
\[
V^\vee=\bigoplus_{a\in \ZZ/\ell\ZZ} W_a,
\qquad
\dim W_a=n_a,
\]
where the linear forms in \(W_a\) have weight \(a\).

For \(t\in \ZZ/\ell\ZZ\), put $q_t:=\dim(S_2)_t$.

We now compute \(q_t\) explicitly. A quadratic monomial of weight \(t\) is a product of
two linear forms whose weights \(a,b\) satisfy \(a+b=t\). There are two cases.

First suppose \(a\neq b\). Then the products of one linear form from \(W_a\) and one
from \(W_b\) span a subspace of dimension \(n_an_b\). If we sum over ordered pairs
\((a,b)\) with \(a+b=t\), this contribution is counted twice, once as \((a,b)\) and
once as \((b,a)\).

Second suppose \(a=b\). Then we must have \(2a=t\), and the corresponding contribution
is
\(
\dim \Sym^2 W_a=\binom{n_a+1}{2}=\frac{n_a^2+n_a}{2}.
\)
Hence
\[
q_t
=
\frac12\sum_{\substack{a+b=t\\ a\neq b}}n_an_b
+
\sum_{2a=t}\binom{n_a+1}{2}.
\]
Multiplying by \(2\), we get
\[
\begin{aligned}
2q_t
&=
\sum_{\substack{a+b=t\\ a\neq b}}n_an_b
+
\sum_{2a=t}(n_a^2+n_a) \\
&=
\left(
\sum_{\substack{a+b=t\\ a\neq b}}n_an_b
+
\sum_{2a=t}n_a^2
\right)
+
\sum_{2a=t}n_a \\
&=
\sum_{a+b=t}n_an_b
+
\sum_{2a=t}n_a.
\end{aligned}
\]
Thus
\begin{equation}\label{eq:opqt}
2q_t
=
\sum_{a+b=t}n_an_b+
\sum_{2a=t}n_a.
\end{equation}

Set
\(
C_t:=\sum_{a+b=t}n_an_b=\sum_a n_an_{t-a}.
\)
By Cauchy--Schwarz,
\[
C_t
=
\sum_a n_an_{t-a}
\le
\left(\sum_a n_a^2\right)^{1/2}
\left(\sum_a n_{t-a}^2\right)^{1/2}.
\]
Since the map \(a\mapsto t-a\) is a permutation of \(\ZZ/\ell\ZZ\), we have $\sum_a n_{t-a}^2=\sum_a n_a^2$.
Therefore\\$C_t\le \sum_a n_a^2$.

Because \(\ell\) is odd, the doubling map $a\longmapsto 2a$ is a bijection of \(\ZZ/\ell\ZZ\). Hence, for each fixed \(t\), there is a unique
\(a\in \ZZ/\ell\ZZ\) with \(2a=t\). Consequently, $\sum_{2a=t}n_a\le n_{\max}$.

Applying these two estimates to \eqref{eq:opqt}, we obtain, for every \(t\),
\(
2q_t
\le
\sum_a n_a^2+n_{\max}.
\)
Taking \(t\) such that \(q_t=\omega_2\), this gives
\(
2\omega_2\le \sum_a n_a^2+n_{\max}.
\)

On the other hand,
\(
E_2=\dim S_2=\binom{m+1}{2},
\)
so
\(
2E_2=m(m+1)=m^2+m.
\)
Therefore
\[
\begin{aligned}
2\gamma_2
&=
2(E_2-\omega_2) =
2E_2-2\omega_2 \\
&\ge
(m^2+m)-\left(\sum_a n_a^2+n_{\max}\right) \\
&=
\left(m^2-\sum_a n_a^2\right)+(m-n_{\max}).
\end{aligned}
\]
By the conjugacy-class dimension formula \eqref{eq:Dg},
\(
D_g=m^2-\sum_a n_a^2.
\)
Hence
\(
2\gamma_2\ge D_g+(m-n_{\max}),
\)
which is equivalent to
\(
\gamma_2\ge \frac12D_g+\frac12(m-n_{\max}).
\)
This proves the lemma.
\end{proof}

\begin{lemma}\label{lem:opq2}
Let
\(
m:=\dim V=n+1, n_{\max}:=\max_a n_a
\)
where \(n_{\max}\) is the largest eigenspace multiplicity of \(g\) on \(V\). Then
\begin{equation}\label{eq:Dgn}
D_g\ge 2n_{\max}(m-n_{\max}).
\end{equation}
\end{lemma}

\begin{proof}
By \eqref{eq:Dg},
\(
D_g=m^2-\sum_a n_a^2.
\)
Since \(m=\sum_a n_a\), we can expand \(m^2\) as
\[
m^2
=
\left(\sum_a n_a\right)^2
=
\sum_a n_a^2+\sum_{\substack{a,b\\ a\neq b}}n_an_b.
\]
Therefore
\[
D_g
=
m^2-\sum_a n_a^2
=
\sum_{\substack{a,b\\ a\neq b}}n_an_b.
\]
This expression is a sum over ordered pairs of distinct weights.

Choose an index \(a_0\) such that$n_{a_0}=n_{\max}$.We separate from the above ordered-pair sum all terms in which one of the two indices is
\(a_0\). Then
\[
\begin{aligned}
D_g
&=
\sum_{\substack{a,b\\ a\neq b}}n_an_b =
\sum_{b\neq a_0}n_{a_0}n_b
+
\sum_{a\neq a_0}n_an_{a_0}
+
\sum_{\substack{a,b\neq a_0\\ a\neq b}}n_an_b.
\end{aligned}
\]
The first two sums are equal, and each is
\(
n_{\max}\sum_{b\neq a_0}n_b
=
n_{\max}(m-n_{\max}).
\)
Hence
\[
D_g
=
2n_{\max}(m-n_{\max})
+
\sum_{\substack{a,b\neq a_0\\ a\neq b}}n_an_b.
\]
The last sum is nonnegative, because all \(n_a\) are nonnegative integers. Therefore\\
\(
D_g\ge 2n_{\max}(m-n_{\max}),
\)
as required.
\end{proof}

\begin{proposition}\label{prop:opq}
Let \(X\subset \PP^n_k\) be a general complete intersection of type \((2^r)\) with\\ \(r\ge 3,  n-r\ge 1\). Then \(X\) admits no nontrivial linear automorphism of odd
prime order \(\ell\neq p\).
\end{proposition}

\begin{proof}
By \Cref{prop:finite-reduction}, it is enough to treat the finitely many relevant odd
prime-order conjugacy classes. Fix one such class and, using \Cref{rem:pgl}, represent it
by a normalized semisimple element \(g\in\GL(V)\). Let \(m:=n+1\), and let
$
n_{\max}:=\max_a n_a$,
$
\rho:=m-n_{\max}.
$
For this fixed class, \Cref{prop:fixed-incidence} and \Cref{prop:inc}, more precisely
\eqref{eq:inccrit}, reduce the statement to
\begin{equation}\label{eq:opqtar}
r\gamma_2>D_g+r(r-1).
\end{equation}
By \eqref{eq:opqgap}, we have
\[
r\gamma_2-D_g-r(r-1)
\ge
\left(\frac r2-1\right)D_g+\frac r2\rho-r(r-1).
\]

We split cases according to \(\rho\).

If \(\rho\ge2\), then \eqref{eq:Dgn} gives
\[
\begin{aligned}
r\gamma_2-D_g-r(r-1)
&\ge
(r-2)(m-\rho)\rho+\frac r2\rho-r(r-1) \\
&\ge
(r-2)(m-1)+r-r(r-1).
\end{aligned}
\]
Here we used two elementary inequalities: \(\rho(m-\rho)\ge m-1\) for
\(2\le\rho\le m-1\), and \(\frac r2\rho\ge r\) for \(\rho\ge2\). The first follows because
\(\rho(m-\rho)\) is concave in \(\rho\), so on the interval \([2,m-1]\) its minimum occurs
at an endpoint, where it is at least \(m-1\). Since \(n-r\ge1\), we have \(m\ge r+2\).
Hence
\[
r\gamma_2-D_g-r(r-1)
\ge
(r-2)(r+1)+r-r(r-1)
=
r-2>0.
\]

If \(\rho=1\), then the eigenspace multiplicities are \((m-1,1)\). Thus \(D_g=2m-2\).
After translating weights, let \(A\) be the \((m-1)\)-dimensional eigenspace of weight
\(0\), and let the remaining one-dimensional eigenspace have weight \(\alpha\neq0\).
Since \(\ell\) is odd, the weights \(0\), \(\alpha\), and \(2\alpha\) are distinct. The
largest eigenspace of \(\Sym^2(V^\vee)\) is therefore \(\Sym^2A\), and \(\gamma_2=m\).
Therefore
\[
\begin{aligned}
r\gamma_2-D_g-r(r-1)
&=rm-(2m-2)-r(r-1) \\
&=(r-2)m+2-r(r-1) \\
&\ge (r-2)(r+2)+2-r(r-1) \\
&=r-2>0.
\end{aligned}
\]
This proves \eqref{eq:opqtar}. Hence \Cref{prop:fixed-incidence} makes the stabilized
locus for this fixed conjugacy class proper. The relevant classes are finite by
\Cref{prop:finite-reduction}, and their finite union is proper by the finite-union
assertion in \Cref{prop:fixed-incidence}. This proves the proposition.
\end{proof}

\subsubsection{The quadratic case: involutions}

\begin{proposition}\label{prop:inv}
Assume \(\operatorname{char}k\neq 2\). Let \(X\subset \PP^n_k\) be a general complete
intersection of type \((2^r)\) with \(r\ge 3, n-r\ge 1\). Then \(X\) admits no
nontrivial involution.
\end{proposition}

\begin{proof}
Fix a nontrivial involution class in \(\PGL(V)\), represented by \(\sigma\). By
\Cref{rem:pgl}, choose a normalized lift, still denoted \(\sigma\), in \(\GL(V)\), so
that \(\sigma^2=I\). Because \(\operatorname{char}k\neq2\), the polynomial \(x^2-1\) has
distinct roots. Hence \(\sigma\) is semisimple with eigenvalues contained in
\(\{1,-1\}\). Thus \(V=V_+\oplus V_-\), where \(V_+\) and \(V_-\) are the \(+1\)- and
\(-1\)-eigenspaces of \(\sigma\), respectively. Write
$\dim V_+=n_1$, $\dim V_-=n_2$, $n_1+n_2=n+1$.
Since the class of \(\sigma\) in \(\PGL(V)\) is nontrivial, both eigenspaces are
nonzero.

For \(S_2=\Sym^2(V^\vee)\), the induced action has the decomposition
\(S_2=(S_2)_+\oplus(S_2)_-\), where
\[
(S_2)_+=\Sym^2(V_+^\vee)\oplus\Sym^2(V_-^\vee),
\qquad
(S_2)_-=V_+^\vee\otimes V_-^\vee.
\]
Set
\[
\begin{aligned}
q_+&:=\dim(S_2)_+
=
\binom{n_1+1}{2}+\binom{n_2+1}{2}, \\
q_-&:=\dim(S_2)_-=n_1n_2.
\end{aligned}
\]

If \(U\subset S_2\) is a \(\sigma\)-stable \(r\)-plane, then
\[
U=U_+\oplus U_-,
\qquad
U_+:=U\cap(S_2)_+,
\qquad
U_-:=U\cap(S_2)_-.
\]
Write
\(
f_+:=\dim U_+,
\
f_-:=\dim U_-,
\
f_++f_-=r.
\)
For fixed \((f_+,f_-)\), the corresponding locus is parametrized by
\(
\Gr(f_+,(S_2)_+)\times\Gr(f_-,(S_2)_-),
\)
and hence has dimension \\\(f_+(q_+-f_+)+f_-(q_--f_-)\). Also
\(
\dim\Cl_{\GL(V)}(\sigma)=2n_1n_2=2q_-.
\)
Therefore, after allowing \(\sigma\) to vary in its conjugacy class, the codimension in
\(\Gr(r,S_2)\) is at least
\[
\begin{aligned}
&r(q_++q_- -r)
-\bigl(f_+(q_+-f_+)+f_-(q_--f_-)+2q_-\bigr) \\
&\qquad = f_-(q_+-f_+)+f_+(q_--f_-)-2q_-.
\end{aligned}
\]
Writing \(f_-=u\) and \(f_+=r-u\), this lower bound is
\[
\Phi(u):=u(q_+-r+u)+(r-u)(q_--u)-2q_-,
\qquad
0\le u\le r.
\]
It remains to show that \(\Phi(u)>0\) for all \(u\). Put \(m:=n_1+n_2=n+1\). Then \(m\ge
r+2\), and since \(n_1,n_2>0\), \(q_-=n_1n_2\ge m-1\ge r+1\). Moreover
\[
q_+-q_-
=
\frac{(n_1-n_2)^2+m}{2}
\ge
\frac m2
\ge
\frac{r+2}{2}.
\]
Using
\(
\Phi(u)=(r-2)q_-+u(q_+-q_- -2r)+2u^2,
\)
we get
\[
\Phi(u)
\ge
(r-2)(r+1)-\frac{3r-2}{2}u+2u^2.
\]
The right-hand side, viewed as a quadratic polynomial in the real variable \(u\), has
minimum
\[
(r-2)(r+1)-\frac{(3r-2)^2}{32}
=
\frac{23r^2-20r-68}{32},
\]
which is strictly positive for every \(r\ge3\): at \(r=3\) the numerator is \(79\), and it
is increasing for \(r\ge3\). Hence \(\Phi(u)>0\) for all \(0\le u\le r\), and therefore
\(\dim Z_\sigma(r,S_2)<\dim\Gr(r,S_2)\). By \Cref{prop:fixed-incidence}, the stabilized
locus for this fixed involution class is contained in a proper closed subset. There are
only finitely many involution classes, indexed by \((n_1,n_2)\) up to interchange. Hence,
again by the finite-union part of \Cref{prop:fixed-incidence}, the union of all
involution-stabilized loci is still proper. This proves the proposition.
\end{proof}

\begin{proof}[Proof of \Cref{thm:eq}]
If $d\ge 3$, the claim follows from \cref{prop:d3}. Now assume $d=2$ and $r\ge 3$.
Prime-to-characteristic automorphisms of odd prime order are excluded by
\cref{prop:opq}. If the prime order is $2$, then necessarily \(\operatorname{char}k\neq
2\), and this case is excluded by \cref{prop:inv}. Thus, in every admissible
equal-degree case, there is no nontrivial prime-to-characteristic linear automorphism of
prime order.
\end{proof}

\subsection{The mixed-degree case}\label{sec:mix}

Throughout this subsection, if \(M\) is a finite-dimensional vector space equipped with
a semisimple action of \(g\), we set
\[
\omega(M):=\max_\lambda \dim M_\lambda,
\qquad
\gamma(M):=\dim M-\omega(M).
\]
For a graded quotient \(R\), we write
\(
\gamma_e(R):=\dim R_e-\max_\lambda\dim(R_e)_\lambda,
\)
and for a fixed semisimple element \(g\) we write \(D_g:=\dim\Cl_{\GL(V)}(g)\).
Whenever \(g\) comes from a projective automorphism of prime order \(\ell\neq p\), we use
\Cref{rem:pgl} to replace it by a normalized lift to \(\GL(V)\); the number \(D_g\) then
also equals the corresponding projective conjugacy-class dimension.

\begin{lemma}\label{lem:qhalf}
Let $g\in \GL(V)$ be a semisimple element of prime order $\ell\neq p$ whose image in
$\PGL(V)$ is nontrivial, and set \(D_g:=\dim \Cl_{\GL(V)}(g)\). With the notation
introduced above, namely
\[
\begin{aligned}
E_2&:=\dim S_2=\binom{n+2}{2}, \\
\omega_2&:=\max_t\dim(S_2)_t,
\qquad
\gamma_2:=E_2-\omega_2,
\end{aligned}
\]
one has
\begin{equation}\label{eq:qhalf}
\gamma_2\ge \frac12D_g.
\end{equation}
Moreover, if $\ell$ is odd, then
\begin{equation}\label{eq:qodd}
\gamma_2\ge \frac12D_g+1.
\end{equation}
\end{lemma}

\begin{proof}
Choose a primitive \(\ell\)-th root of unity \(\zeta\). Decompose \(V\) into
\(g\)-eigenspaces
\[
V=\bigoplus_{a\in \mathbb Z/\ell\mathbb Z}V_a,
\qquad
V_a:=\{v\in V: g v=\zeta^a v\}.
\]
Set
\(
n_a:=\dim V_a,
\
n_{\max}:=\max_a n_a.
\)
If $\ell$ is odd, \cref{lem:opq1} gives
\(
\gamma_2\ge \frac12D_g+\frac12\bigl((n+1)-n_{\max}\bigr).
\)
Since the image of $g$ in $\PGL(V)$ is nontrivial, $n_{\max}\le n$. Thus the second term
is positive. Since $\gamma_2$ and $\frac12D_g$ are integers, this gives \(\gamma_2\ge
\frac12D_g+1\). If $\ell=2$, write
\(
V=V_+\oplus V_-,
\
\\\dim V_+=n_1,\ \dim V_-=n_2.
\)
Then \(D_g=2n_1n_2, \gamma_2=n_1n_2\). Hence
\begin{equation}\label{eq:qinv}
\gamma_2=\frac12D_g.
\end{equation}
\end{proof}

\begin{lemma}\label{lem:qqgap}
Let \(Q\in S_2\) be a smooth \(g\)-eigenquadric, meaning that \(g\cdot Q=\chi Q\) for
some \(\chi\in k^\times\), and assume \(\dim V=n+1\ge 4\). Set \(R:=S/(Q)\). Then, for
every \(d\ge 3\),
\begin{equation}\label{eq:qqgap}
\gamma_d(R)\ge \gamma_2-1.
\end{equation}
Moreover, if \(\ell=2\), then
\begin{equation}\label{eq:qqinv}
\gamma_d(R)\ge \gamma_2+1.
\end{equation}
\end{lemma}

\begin{proof}
Since \(R_2=S_2/\langle Q\rangle\) and $\langle Q\rangle$ is one-dimensional and
$g$-stable, we have \(\gamma_2(R)\ge \gamma_2-1\). The first assertion follows from
\cref{lem:qgapmon}.

Assume now $\ell=2$. Write
\[
V^\vee=V_+\oplus V_-,
\qquad
\dim V_+=a,\quad \dim V_-=b.
\]
Since the image of $g$ in $\PGL(V)$ is nontrivial, both eigenspaces are nonzero, so
\(a,b>0\). Moreover, by assumption, \(a+b=\dim V\ge 4\).

 Suppose $Q$ has weight $+$. Recall that $\gamma_2$ is the dimension of the complement of the largest eigenspace in $S_2$. The induced action on $S_2$ has two eigenspaces: the positive weight space $(S_2)_+ = \Sym^2(V_+) \oplus \Sym^2(V_-)$ and the negative weight space $(S_2)_- = V_+ \otimes V_-$. Their dimensions are given by
\[
\dim (S_2)_+ = \binom{a+1}{2} + \binom{b+1}{2}, \qquad \dim (S_2)_- = ab.
\]
The difference between their dimensions is exactly
\[
\dim (S_2)_+ - \dim (S_2)_- = \frac{a^2+a+b^2+b-2ab}{2} = \frac{(a-b)^2+a+b}{2}.
\]
Since $a,b > 0$, this difference is strictly positive, meaning $\dim (S_2)_+ > \dim (S_2)_-$. Hence, the largest eigenspace is $(S_2)_+$, which yields $\gamma_2 = \dim (S_2)_- = ab$. 

In degree $3$, \(R_3=S_3/(QS_1)\). The two
eigenspaces of $R_3$ have dimensions
\[
D_1:=\binom{a+1}{2}b+\binom{b+2}{3}-b,
\qquad
D_2:=\binom{a+2}{3}+a\binom{b+1}{2}-a.
\]
We check explicitly that both are at least \(ab+1\). Indeed,
\[
D_1-(ab+1)
=
\frac{ab(a-1)}2+\frac{b(b-1)(b+4)}6-1.
\]
This is nonnegative as follows: if \(b\ge2\), then the second term is at least \(1\),
while the first term is nonnegative; if \(b=1\), then \(a\ge3\), and the expression is
\(a(a-1)/2-1\ge2\). Similarly,
\[
D_2-(ab+1)
=
\frac{a(a-1)(a+4)}6+\frac{ab(b-1)}2-1.
\]
If \(a\ge2\), then the first term is at least \(2\), while the second term is
nonnegative; if \(a=1\), then \(b\ge3\), and the expression is
\(b(b-1)/2-1\ge2\). Hence \(D_2\ge ab+1\) as well. Thus
\(\gamma_3(R)\ge \gamma_2+1\). 

If $Q$ has weight $-$, then
\(Q\in V_+\otimes V_-\)defines a perfect pairing between the two eigenspaces, so smoothness
forces $a=b$. Since $a+b\ge 4$, we have $a=b\ge 2$. Then \(\gamma_2=a^2\), and both
eigenspaces of $R_3$ have dimension
\(
D:=\binom{a+2}{3}+a\binom{a+1}{2}-a.
\)
Moreover
\(
D-(a^2+1)=\frac{2a^3-2a-3}{3}>0
\qquad (a\ge2).
\)
Thus \(\gamma_3(R)\ge \gamma_2+1\). Again by \cref{lem:qgapmon}, the same inequality holds
for all $d\ge 3$.
\end{proof}

\begin{lemma}\label{lem:hqgap}
Let $g\in \GL(V)$ be as above, and set \(D_g:=\dim \Cl_{\GL(V)}(g)\). Assume \(\dim
V\ge4\). Let $a\ge 3$, let $F\in S_a$ be a smooth $g$-eigenform, meaning that \(g\cdot
F=\chi F\) for some \(\chi\in k^\times\), and set \(R:=S/(F)\). Then for every $b>a$,
\begin{equation}\label{eq:hqgap}
\gamma_b(R)\ge \frac12D_g+1.
\end{equation}
\end{lemma}

\begin{proof}
By the monotonicity \eqref{eq:qgapmon}, \(\gamma_b(R)\ge \gamma_a(R)\). Since
\(R_a=S_a/\langle F\rangle\) and $\langle F\rangle$ is one-dimensional and $g$-stable,
\(\gamma_a(R)\ge \gamma_a-1\). By \eqref{eq:d3gap}, \(\gamma_a>\frac12D_g+1\). Since
both $\gamma_a$ and $\frac12D_g$ are integers, \(\gamma_a\ge \frac12D_g+2\). Therefore
\(\gamma_b(R)\ge \frac12D_g+1\). \end{proof}

\begin{remark}\label{rem:2step}
In the mixed-degree estimates below we use the following fixed-\(g\) two-step count.
Suppose the lower-degree condition contributes codimension at least \(A\). After the
lower-degree equations are fixed, suppose the next minimal equation space is an
\(r\)-plane in a semisimple \(g\)-module \(M\). If \(\gamma(M)\ge B\), then the
condition that this \(r\)-plane be \(g\)-stable contributes codimension at least
\(rB-r(r-1)\), by the fixed-\(g\) form of the proof of \Cref{prop:inc}. Thus the total
codimension contribution is at least \(A+rB-r(r-1)\). 
\end{remark}

\begin{theorem}\label{prop:mix}
Let \(X\subset \PP^n_k\) be a general smooth complete intersection whose multidegree
begins with one of the following mixed-degree patterns:
\begin{enumerate}[label=\textup{(\arabic*)}]
 \item $(2,2,d^r)$ with $d\ge 3$, $r\ge 1$, and $n-(2+r)\ge 1$;
 \item $(2,d^r)$ with $d\ge 3$, $r\ge 1$, and $n-(1+r)\ge 1$;
 \item $(a,b^r)$ with $3\le a<b$, $r\ge 1$, and $n-(1+r)\ge 1$.
\end{enumerate}
Then $X$ admits no nontrivial linear automorphism of prime order $\ell\neq p$.
\end{theorem}

\begin{proof}
By \Cref{prop:finite-reduction,prop:fixed-incidence}, it is enough to show,
for each of the finitely many relevant nontrivial prime-order conjugacy
classes, that the corresponding stabilized locus is proper. By
\cref{lem:fl,lem:new,prop:bt}, in each of the three mixed-degree cases it is
enough to consider the first one or two distinct degrees: every linear
automorphism preserves the lower-degree part and acts on the minimal equation
space in the next degree, while any remaining higher-degree equations impose
only additional conditions.

Whenever a conjugacy class is fixed below, we use \Cref{rem:pgl} to represent
it by a normalized semisimple element \(g\in\GL(V)\), and write
\(
D_g:=\dim\Cl_{\GL(V)}(g).
\)

\smallskip
\noindent\emph{Case 1: Type \((2,2,d^r)\).}
It is enough to work over the nonempty open locus where the
two-dimensional quadratic equation space cuts out a smooth complete
intersection of two quadrics. Let
\(
\Lambda\subset S_2
\)
be this quadratic equation space, and set
\(
Y:=V(\Lambda)\subset\PP^n,
\
A(Y):=S/I_Y.
\)
By \Cref{lem:fl}, every linear automorphism of \(X\) preserves
\(\Lambda\), and hence preserves \(Y\).

By the finiteness statement recalled in the incidence setup,
\(\Aut_L(Y)\) is finite. By \Cref{prop:bt}, after fixing \(Y\), the
degree-\(d\) minimal equation spaces giving smooth extensions contain a
nonempty open subset of
\(
\Gr\bigl(r,A(Y)_d\bigr).
\)
Moreover,
\(
r<\dim A(Y)_d.
\)
Indeed, \(Y\) is a positive-dimensional smooth complete intersection,
hence integral. If \(0\neq z\in A(Y)_1\), multiplication by \(z^{d-1}\)
gives an injection
\(
A(Y)_1\hookrightarrow A(Y)_d,
\)
and therefore
\(
\dim A(Y)_d\ge n+1>r.
\)

Let \(T(Y)\) be the finite set of nontrivial elements of prime order
different from \(p\) in \(\Aut_L(Y)\). We claim that every
\(\sigma\in T(Y)\) acts non-scalarly on \(A(Y)_d\). Since \(\sigma\)
has order prime to \(p\), its action is semisimple. As \(\sigma\) is
nontrivial in \(\PGL(V)\), there exist eigenvectors
\(
u,v\in A(Y)_1
\)
with distinct eigenvalues \(\alpha\neq\beta\). Since \(A(Y)\) is a
domain, the elements
\(
u^d,\ u^{d-1}v
\)
are nonzero, and have eigenvalues
\(
\alpha^d,\ \alpha^{d-1}\beta,
\)
respectively. These are distinct, so \(\sigma\) does not act as a
scalar on \(A(Y)_d\).

By \Cref{lem:kill}, for every fixed \(Y\), a general \(r\)-plane in
\(A(Y)_d\) is stabilized by no element of \(T(Y)\). Finally, the
relative linear automorphism incidence is finite over the smooth
two-quadric locus. Hence the corresponding invariant-subspace
incidence in the relative Grassmannian has dimension strictly smaller
than that of the whole relative Grassmannian. Together with
\Cref{prop:bt}, this shows that a general smooth extension admits no
nontrivial linear automorphism of prime order different from \(p\).

\smallskip
\noindent\emph{Case 2: Type \((2,d^r)\).}
Fix a relevant conjugacy class, with normalized representative \(g\) and
orbit dimension \(D_g\) as above. Let \(Q\) be the quadratic equation and set
\(
R:=S/(Q).
\)
For the general member, \(Q\) is smooth.

For fixed \(g\), requiring the quadratic line
\(
\langle Q\rangle\subset S_2
\)
to be \(g\)-stable contributes codimension at least \(\gamma_2\). Once such
a smooth \(g\)-eigenquadric is fixed, the degree-\(d\) minimal equation space
is an \(r\)-plane in
\(
R_d=S_d/(QS_{d-2}),
\)
and by \Cref{rem:2step} its \(g\)-stability contributes codimension at least
\(
r\gamma_d(R)-r(r-1).
\)
Thus it is enough to prove
\begin{equation}\label{eq:mix2d}
\gamma_2+r\gamma_d(R)-r(r-1)>D_g.
\end{equation}

Suppose first that \(\ell\) is odd. By \eqref{eq:qodd} and
\eqref{eq:qqgap},
\[
\gamma_2\ge\frac12D_g+1,
\qquad
\gamma_d(R)\ge\gamma_2-1\ge\frac12D_g.
\]
Hence
\[
\begin{aligned}
&\gamma_2+r\gamma_d(R)-r(r-1)-D_g\\
&\qquad\ge
\frac{r-1}{2}D_g+1-r(r-1).
\end{aligned}
\]
Since \(D_g\ge2n\) by \eqref{eq:Dglb}, and
\(
n-(1+r)\ge1
\quad\Longrightarrow\quad
n\ge r+2,
\)
we obtain
\[
\begin{aligned}
\gamma_2+r\gamma_d(R)-r(r-1)-D_g
&\ge
(r-1)(r+2)+1-r(r-1)\\
&=
2r-1>0.
\end{aligned}
\]

If \(\ell=2\), then by \eqref{eq:qinv} and \eqref{eq:qqinv},
\[
\gamma_2=\frac12D_g,
\qquad
\gamma_d(R)\ge\frac12D_g+1.
\]
Therefore
\[
\begin{aligned}
&\gamma_2+r\gamma_d(R)-r(r-1)-D_g\\
&\qquad\ge
\frac{r-1}{2}D_g+r-r(r-1)\\
&\qquad\ge
(r-1)(r+2)+r-r(r-1)\\
&\qquad=
3r-2>0.
\end{aligned}
\]
Thus \eqref{eq:mix2d} holds in both cases, and the stabilized locus associated
with the fixed conjugacy class is proper.

\smallskip
\noindent\emph{Case 3: Type \((a,b^r)\) with \(3\le a<b\).}
This case already follows from the classical theorem of Matsumura--Monsky
\cite[Theorem~5]{MatsumuraMonsky1963} on the generic triviality of the linear
automorphism group of hypersurfaces of degree at least \(3\). We nevertheless
give a direct self-contained proof in the same incidence-theoretic framework.

Fix a relevant conjugacy class, with normalized representative \(g\) and
orbit dimension \(D_g\) as above. Let \(F\) be the degree-\(a\) equation and
set
\(
R:=S/(F).
\)
For the general member, \(F\) is smooth.

For fixed \(g\), requiring the line
\(
\langle F\rangle\subset S_a
\)
to be \(g\)-stable contributes codimension at least \(\gamma_a\). Once such
a smooth \(g\)-eigenform is fixed, the degree-\(b\) minimal equation space is
an \(r\)-plane in
\(
R_b=S_b/(FS_{b-a}),
\)
and by \Cref{rem:2step} its \(g\)-stability contributes codimension at least
\(
r\gamma_b(R)-r(r-1).
\)
It is therefore enough to prove
\begin{equation}\label{eq:mixab}
\gamma_a+r\gamma_b(R)-r(r-1)>D_g.
\end{equation}

By \eqref{eq:d3gap}, together with the integrality stated there, and by
\eqref{eq:hqgap},
\[
\gamma_a\ge\frac12D_g+2,
\qquad
\gamma_b(R)\ge\frac12D_g+1.
\]
Consequently,
\[
\begin{aligned}
&\gamma_a+r\gamma_b(R)-r(r-1)-D_g\\
&\qquad\ge
\frac{r-1}{2}D_g+r+2-r(r-1).
\end{aligned}
\]
Using again \(D_g\ge2n\) and \(n\ge r+2\), we get
\[
\begin{aligned}
\gamma_a+r\gamma_b(R)-r(r-1)-D_g
&\ge
(r-1)(r+2)+r+2-r(r-1)\\
&=
3r>0.
\end{aligned}
\]
Thus \eqref{eq:mixab} holds, and the stabilized locus associated with the
fixed conjugacy class is proper.

We have therefore obtained a proper stabilized locus for every relevant
prime-order conjugacy class in each of the three cases. Since only finitely
many such classes occur by \Cref{prop:finite-reduction}, their union is
proper by the finite-union assertion in \Cref{prop:fixed-incidence}. Hence a
general complete intersection of any of the stated mixed types admits no
nontrivial linear automorphism of prime order \(\ell\neq p\).
\end{proof}

\begin{proof}[Proof of \Cref{thm:ptp}]
By \Cref{thm:eq,prop:mix}, all possible initial degree patterns are covered,
except for \((2,2)\), which is excluded by hypothesis. Hence a general \(X\)
admits no nontrivial linear automorphism of prime order \(\ell\neq p\).

If \(\sigma\) had finite order \(m>1\) prime to \(p\), then for any prime
divisor \(\ell\mid m\), the element \(\sigma^{m/\ell}\) would have order
\(\ell\neq p\), a contradiction. Thus \(X\) admits no nontrivial
prime-to-\(p\) linear automorphism.
\end{proof}

\section{Order-\texorpdfstring{$p\neq2$}{p} automorphisms}\label{sec:op}

This section proves \Cref{thm:op}. Throughout this section \(p=\operatorname{char}k\ge
3\). In the geometric applications below we are in the range \(c\ge2,  n-c\ge1\),
so \(n\ge3, \\m:=\dim V=n+1\ge4\), where $V$ is the vector space defining the ambient projective space \(\PP^n_k\). The order-\(p\) automorphisms considered here
are unipotent. We shall use the reductions from \Cref{sec:fl}: a linear automorphism
preserves the lowest-degree part of the ideal, and in the mixed-degree cases it also
acts on the minimal equation space in the next degree.

\subsection{Unipotent incidence}\label{sec:opinc}

\begin{proposition}\label{prop:unip}
Let \(g\in \Aut_L(X)\) have order \(p\). After choosing a lift to \(\GL(V)\) and
rescaling it, we may write \(g=I+N,  N\neq0,N^p=0\). For every
finite-dimensional induced \(k[g]\)-module \(M\), set \(\Delta_M:=g_M-\id_M\). Then
\(\Delta_M^p=0\). In particular, on \(S=\Sym(V^\vee)\) we write \(\Delta:=g^*-\id\).
\end{proposition}

\begin{proof}
The image of \(g\) in \(\PGL(V)\) has order \(p\). If a lift satisfies \(g^p=\lambda
I\), then after rescaling by a \(p\)-th root of \(\lambda^{-1}\), we may assume
\(g^p=I\). Since \(x^p-1=(x-1)^p\) in characteristic \(p\), this lift is unipotent, and
\((g-I)^p=g^p-I=0\). The same argument applies to all induced modules.
\end{proof}

For a finite-dimensional \(k[g]\)-module \(M\), put \(\Omega(M):=\dim\Img(\Delta_M)\).\\
If \(M=S_d=\Sym^d(V^\vee)\), write \(\Omega_d(V):=\Omega(S_d)\).

We record the basic meaning of this number. Since \(g=\id+\Delta_M\), a subspace
\(U\subset M\) is \(g\)-stable if and only if \(\Delta_M(U)\subset U\). Thus the problem
of imposing invariance under \(g\) is equivalent to imposing that \(U\) be stable under
the nilpotent operator \(\Delta_M\).

In particular, if \(L\subset M\) is a line, then \(L\) is \(g\)-stable if and only if
\(L\subset \Ker(\Delta_M)\). Indeed, the induced action of the unipotent element \(g\)
on the one-dimensional space \(L\) must be trivial. Hence the locus of \(g\)-stable
lines in \(\mathbb P(M)\) is precisely \(\mathbb P(\Ker\Delta_M)\subset \mathbb P(M)\),
and its codimension is
\(
\dim M-\dim\Ker(\Delta_M)
=
\dim\Img(\Delta_M)
=
\Omega(M).
\)
Thus \(\Omega(M)\) is exactly the number of independent linear conditions imposed by
asking a general line in \(M\) to be \(g\)-stable.

For higher-dimensional subspaces \(U\in\Gr(s,M)\), the condition is likewise
\(\Delta_M(U)\subset U\). Equivalently, the induced map \(U\longrightarrow M/U\)
obtained from \(\Delta_M\) must vanish. The rank \(\Omega(M)=\dim\Img(\Delta_M)\)
measures the size of this condition. The next lemma gives the precise Grassmannian
codimension estimate needed below.

\begin{lemma}\label{lem:stgr}
Let \(M\) be a finite-dimensional vector space with a nilpotent endomorphism \(\Delta\),
and set \(\Omega_M:=\dim\Img\Delta\). Then the locus \(\{U\in\Gr(s,M):\Delta U\subset
U\}\) has codimension at least
\begin{equation}\label{eq:stgr}
 s\Omega_M-s(s-1)
\end{equation}
in \(\Gr(s,M)\).
\end{lemma}

\begin{proof}
Consider flags
\[
 0=U_0\subset U_1\subset\cdots\subset U_s\subset M,
 \qquad \dim U_i=i,
\]
satisfying \(\Delta U_i\subset U_{i-1}\). At the \(i\)-th step, after \(U_{i-1}\) is
fixed, the possible choices for \(U_i/U_{i-1}\) lie in \(\Delta^{-1}(U_{i-1})/U_{i-1}\).
Since
\[
\begin{aligned}
 \dim \Delta^{-1}(U_{i-1})
 &= \dim\Ker\Delta+
 \dim\bigl(U_{i-1}\cap\Img\Delta\bigr) \\
 &\le (\dim M-\Omega_M)+(i-1)
 = \dim M-\Omega_M+i-1,
\end{aligned}
\]
the codimension in the full flag variety is at least
\[
 \sum_{i=1}^s (\Omega_M-i+1)=s\Omega_M-\binom{s}{2}.
\]
Every \(\Delta\)-stable \(s\)-plane admits such an adapted flag. Projecting to
\(\Gr(s,M)\) can lose at most the flag-fiber dimension \(\binom{s}{2}\), so the
codimension in the Grassmannian is at least
\[
 s\Omega_M-\binom{s}{2}-\binom{s}{2}=s\Omega_M-s(s-1).
\]
\end{proof}

\begin{remark}\label{rem:uinc}
We will use \eqref{eq:stgr} in the following form. Fix a nontrivial
order-\(p\) unipotent element \(g\in \PGL(V)\), and let
\(M\) be the relevant \(k[g]\)-module of minimal equations. If an
\(s\)-dimensional equation space \(U\subset M\) is preserved by \(g\), then $\Delta_M(U)\subset U$,
  $ \Delta_M:=g_M-\mathrm{id}_M$.
Thus \eqref{eq:stgr} gives a lower bound for the codimension of the
corresponding stabilized locus inside \(\Gr(s,M)\). We then compare this
codimension with the dimension of the conjugacy class of \(g\) in \(\PGL(V)\).
Since there are only finitely many order-\(p\) unipotent conjugacy classes, it
suffices to check this comparison class by class.
\end{remark}

Without changing the relevant conjugacy-class dimension, we may work in \(\GL(V)\), so below we
compute them in \(\GL(V)\) and write \(D:=\dim\Cl_{\GL(V)}(N)\). There are only finitely
many order-\(p\) unipotent conjugacy classes, since they are determined by Jordan block
sizes.

\subsection{Uniform image-rank estimates}\label{sec:opest}

All image-rank calculations in this subsection are carried out for the operator
\(\Delta_M=g_M-\id_M\) on the induced module under consideration. In particular, when
\(M\) is a tensor product or a symmetric power, \(\Delta_M\) is the operator induced by
the group action on that module.

\begin{lemma}\label{lem:ambmonp}
For every \(d\ge1\),
\begin{equation}\label{eq:ambmonp}
 \Omega_d(V)\le \Omega_{d+1}(V).
\end{equation}
\end{lemma}

\begin{proof}
Choose \(0\neq z\in(V^\vee)^g\). Multiplication by \(z\) gives an injective map
\(S_d\hookrightarrow S_{d+1}\) commuting with \(\Delta\), and hence sends
\(\Img(\Delta|S_d)\) injectively into \(\Img(\Delta|S_{d+1})\).
\end{proof}

\begin{lemma}\label{lem:qmonp}
Let \(H\in S_e\), \(e\ge2\), be a nonzero \(g\)-invariant irreducible form, and set \(R:=S/(H)\).
For each \(t\ge0\), write \(R_t\) for the degree-\(t\) graded piece of \(R\), and set
\(
\Omega_t(R):=\dim \operatorname{Im}\bigl(\Delta:R_t\to R_t\bigr).
\)
Then for all \(t\ge e\),
\begin{equation}\label{eq:qmonp}
\Omega_t(R)\le \Omega_{t+1}(R).
\end{equation}
\end{lemma}

\begin{proof}
Since \(R=S/(H)\) is a domain, it is enough to find a nonzero invariant linear form in
\(R_1\). Choose \(0\neq z\in (V^\vee)^g\). Since \(e\ge2\), the image of \(z\) in
\(R_1\) is nonzero. Hence multiplication by \(z\) gives an injective map
\(R_t\hookrightarrow R_{t+1}\). Because \(z\) is \(g\)-invariant, multiplication by
\(z\) commutes with \(\Delta\). Therefore it sends \(\operatorname{Im}(\Delta|_{R_t})\)
injectively into $\operatorname{Im}(\Delta|_{R_{t+1}})$. Thus \(\Omega_t(R)\le \Omega_{t+1}(R)\). \end{proof}

\begin{lemma}\label{lem:qcrp}
Let \(H\in S_e\) be a nonzero \(g\)-invariant form and set \(R=S/(H)\). Then
\begin{equation}\label{eq:qcre}
 \Omega_e(R)\ge \Omega_e(V)-1.
\end{equation}
If \(e=2\), then
\begin{equation}\label{eq:qcr2}
 \Omega_3(R)\ge \Omega_3(V)-\dim V.
\end{equation}
\end{lemma}

\begin{proof}
The quotient map \(S_e\to R_e\) kills the one-dimensional \(g\)-stable space \(\langle
H\rangle\), so the image rank drops by at most \(1\). If \(e=2\), then \(R_3=S_3/(H\cdot
V^\vee)\), and \(\dim(H\cdot V^\vee)=\dim V\). Thus the image rank in degree \(3\) drops
by at most \(\dim V\).
\end{proof}

For \(1\le a\le p\), write \(J_a\) for the indecomposable \(k[g]\)-module of dimension
\(a\), equivalently one Jordan block of length \(a\) for \(\Delta=g-\id\). Thus
\(J_a\simeq k[s]/(s^a)\), with \(g\) acting by multiplication by \(1+s\).

\begin{lemma}\label{lem:tensp}
For finite-dimensional \(k[g]\)-modules \(A\) and \(B\),
\begin{equation}\label{eq:tenslb}
 \Omega(A\otimes B)\ge \Omega(A)\dim B.
\end{equation}
More precisely, if \(A\simeq\bigoplus_iJ_{a_i}\) and \(B\simeq\bigoplus_jJ_{b_j}\), then
\begin{equation}\label{eq:tens}
 \Omega(A\otimes B)=\dim A\dim B-
 \sum_{i,j}\min(a_i,b_j).
\end{equation}
\end{lemma}

\begin{proof}
It is enough to treat one pair of indecomposable blocks. For \(J_a\otimes J_b\), use the
model \(J_a\otimes J_b\simeq k[s,t]/(s^a,t^b)\), where \(g\) acts by multiplication by
\((1+s)(1+t)\). Therefore the actual operator \(\Delta=g-\id\) acts by multiplication by
\((1+s)(1+t)-1=s+t+st\). The cokernel of this multiplication map has length
\(
\ell\bigl(k[s,t]/(s^a,t^b,s+t+st)\bigr).
\)
Since \(1+t\) is a unit, the relation \(s+t+st=0\) is equivalent to
\(s=-\frac{t}{1+t}\). Hence 
\(
k[s,t]/(s^a,t^b,s+t+st) \simeq k[t]/(t^a,t^b),
\)
which has
length \(\min(a,b)\). Thus \(\Omega(J_a\otimes J_b)=ab-\min(a,b)\). Summing over all
pairs of Jordan blocks gives \eqref{eq:tens}. Finally, for each \(i\),
\(
\sum_j\min(a_i,b_j)\le \sum_j b_j=\dim B.
\)
It follows that
\[
\begin{aligned}
\Omega(A\otimes B)
&=
\sum_i\left(a_i\dim B-\sum_j\min(a_i,b_j)\right) \\
&\ge
\sum_i(a_i-1)\dim B =
\Omega(A)\dim B.
\end{aligned}
\]
\end{proof}

\begin{lemma}\label{lem:blk}
For every Jordan block \(J_\lambda\) with \(2\le\lambda\le p\), one has
\[
\begin{aligned}
\Omega_2(J_\lambda)&=\left\lfloor\frac{\lambda^2}{2}\right\rfloor, \\
\Omega_3(J_\lambda)&\ge \frac{(\lambda-1)(\lambda+2)}2.
\end{aligned}
\]
\end{lemma}

\begin{proof}
For degree \(2\), consider
\(
J_\lambda\otimes J_\lambda\simeq k[s,t]/(s^\lambda,t^\lambda).
\)
The operator \(\Delta\) is multiplication by \(u:=s+t+st\). Since \(p\ge3\), the
decomposition
\[
J_\lambda\otimes J_\lambda=\Sym^2J_\lambda\oplus\bigwedge^2J_\lambda
\]
is the decomposition into the \((+1)\)- and \((-1)\)-eigenspaces of the swap involution
\(s\leftrightarrow t\). The element \(u\) is fixed by this involution, so the cokernel
of \(\Delta\) on \(\Sym^2J_\lambda\) is the \((+1)\)-eigenspace of
\(C:=k[s,t]/(s^\lambda,t^\lambda,u)\). As above,
\(
C\simeq k[t]/(t^\lambda),
\
s=-\frac{t}{1+t}.
\)
Under this identification, the swap involution sends \(t\longmapsto -\frac{t}{1+t}\).
Set \(y:=\frac{2t}{2+t}\). This is a change of parameter in \(k[t]/(t^\lambda)\), and
the involution sends \(y\) to \(-y\). Hence the invariant subspace of \(C\) has basis
\(1,y^2,y^4,\dots\), and therefore has dimension \(\lceil\lambda/2\rceil\). It follows
that
\[
\dim\operatorname{coker}\bigl(\Delta|_{\Sym^2J_\lambda}\bigr)
=\left\lceil\frac{\lambda}{2}\right\rceil.
\]
Thus
\[
\Omega_2(J_\lambda)
=
\dim\Sym^2J_\lambda-\left\lceil\frac{\lambda}{2}\right\rceil
=
\left\lfloor\frac{\lambda^2}{2}\right\rfloor.
\]

For degree \(3\), choose a basis \(e_0,e_1,\dots,e_{\lambda-1}\) of \(J_\lambda\) such
that
\[
g(e_i)=e_i+e_{i+1}\quad(0\le i\le\lambda-2),
\qquad
 g(e_{\lambda-1})=e_{\lambda-1}.
\]
Put \(e_\lambda:=0\), and set \(W:=\Span(e_1,\dots,e_{\lambda-1})\). Then \(W\) is
\(g\)-stable. Consider the subspace \(U:=e_0\Sym^2W\oplus e_0^2W\subset
\Sym^3J_\lambda\). We claim that \(\Delta|_U\) is injective. Let
\(
u=e_0f+e_0^2h,
\ \\
f\in\Sym^2W,
\
h\in W,
\)
and suppose \(\Delta(u)=0\), equivalently \(g(u)=u\). Since \(g(e_0)=e_0+e_1\), we have
\(g(u)=(e_0+e_1)g(f)+(e_0+e_1)^2g(h)\). Comparing the powers of \(e_0\) in the
decomposition
\[
\Sym^3J_\lambda
=
\Sym^3W\oplus e_0\Sym^2W\oplus e_0^2W\oplus ke_0^3
\]
gives
\begin{align*}
 g(h)&=h,\\
 g(f)+2e_1h&=f,\\
 e_1g(f)+e_1^2h&=0.
\end{align*}
The last equality is \(e_1\bigl(g(f)+e_1h\bigr)=0\). Since \(\Sym W\) is a polynomial
ring and multiplication by \(e_1\) is injective, we get \(g(f)+e_1h=0\). Together with
\(g(f)+2e_1h=f\), this gives \(f=e_1h\). Substituting back, and using \(g(h)=h\), we
obtain
\[
0=g(f)+e_1h=g(e_1h)+e_1h=(e_1+e_2)h+e_1h=(2e_1+e_2)h.
\]
The linear form \(2e_1+e_2\) is nonzero because \(p\ge3\), and multiplication by a
nonzero linear form in \(\Sym W\) is injective. Hence \(h=0\), and then \(f=0\).
Therefore \(\Delta|_U\) is injective.

Consequently
\[
\begin{aligned}
\Omega_3(J_\lambda)
&\ge \dim U \\
&=\dim\Sym^2W+\dim W \\
&=\binom{\lambda}{2}+\lambda-1 \\
&=\frac{(\lambda-1)(\lambda+2)}2.
\end{aligned}
\]
\end{proof}

\begin{theorem}\label{thm:opunif}
Let \(W\) be a finite-dimensional nontrivial order-\(p\) unipotent \(k[g]\)-module, with
\(m:=\dim W\). Write \(\Delta_W:=g-\id_W\). Thus \(\Delta_W\neq0\) and \(\Delta_W^p=0\).
Let \(D:=\dim\Cl_{\GL(W)}(\Delta_W)\) be the dimension of the \(\GL(W)\)-conjugacy class
of the nilpotent operator \(\Delta_W\).

For \(d\ge1\), set
\[
\Omega_d(W):=\dim\operatorname{Im}\bigl(\Delta:\Sym^d W\to \Sym^d W\bigr),
\]
where \(\Delta=g-\id\) is the actual operator induced on \(\Sym^dW\). Then
\begin{align}
 \Omega_2(W)&\ge m, \label{eq:I2}\\
 \Omega_3(W)&\ge 2m-3, \label{eq:I3}\\
 2\Omega_2(W)-D&\ge 2, \label{eq:2q}\\
 2\Omega_3(W)-D&\ge 2. \notag
\end{align}
Moreover, if \(m\ge4\), then
\begin{align}
 2\Omega_3(W)-D&\ge3, \label{eq:2cs}\\
 \Omega_2(W)+\Omega_3(W)-m&>D. \label{eq:T}
\end{align}
\end{theorem}

\begin{proof}
For any finite-dimensional order-\(p\) unipotent \(k[g]\)-module \(M\), write
\[
D(M):=\dim\Cl_{\GL(M)}(\Delta_M).
\]
Thus, for the module \(W\) in the theorem, one has \(D(W)=D\).

Let \(J_\lambda\) be a largest nontrivial Jordan block of \(W\), where \(\lambda\ge2\),
and write\\
\(
W=J_\lambda\oplus W',
\
m':=\dim W'.
\)
Here \(W'\) is the direct sum of all remaining Jordan blocks, including possible trivial
blocks.

We first prove the lower bounds for \(\Omega_2(W)\) and \(\Omega_3(W)\). Since
\[
\Sym^2W
=
\Sym^2J_\lambda\oplus(J_\lambda\otimes W')\oplus\Sym^2W',
\]
and all summands are \(g\)-stable, image dimensions add over these summands. Therefore,
using \eqref{eq:tenslb} and \Cref{lem:blk},
\[
 \Omega_2(W)
 \ge
 \Omega_2(J_\lambda)+\Omega(J_\lambda\otimes W')
 \ge
 \lambda+(\lambda-1)m'
 \ge
 \lambda+m'
 =
 m.
\]
Here \(\Omega(J_\lambda)=\lambda-1\), and \(\lambda\ge2\).

Similarly,
\[
\Sym^3W
=
\Sym^3J_\lambda
\oplus
(\Sym^2J_\lambda\otimes W')
\oplus
(J_\lambda\otimes \Sym^2W')
\oplus
\Sym^3W'.
\]
Hence
\[
 \Omega_3(W)
 \ge
 \Omega_3(J_\lambda)+\Omega(\Sym^2J_\lambda\otimes W').
\]
By \eqref{eq:tenslb} and \Cref{lem:blk},
\[
\Omega_3(J_\lambda)\ge \frac{(\lambda-1)(\lambda+2)}2,
\qquad
\Omega(\Sym^2J_\lambda\otimes W')\ge \Omega_2(J_\lambda)m'\ge \lambda m'.
\]
Thus
\[
 \Omega_3(W)
 \ge
 \frac{(\lambda-1)(\lambda+2)}2+\lambda m'.
\]
The last expression is at least \(2(\lambda+m')-3\), because
\[
\frac{(\lambda-1)(\lambda+2)}2+\lambda m'-\bigl(2(\lambda+m')-3\bigr)
=
\frac{\lambda^2-3\lambda+4}{2}+(\lambda-2)m'
\ge0
\]
for \(\lambda\ge2\). Hence
\(
 \Omega_3(W)
 \ge
 2(\lambda+m')-3
 =
 2m-3.
\)

We next prove the orbit-dimension estimates. The usual centralizer formula for nilpotent
Jordan modules gives
\[
D(W)=D(J_\lambda)+D(W')+2(\lambda-1)m'.
\]
Indeed, since \(\lambda\) is maximal among the block lengths appearing in \(W\), for
every Jordan block \(J_b\subset W'\) one has
\[
\dim\Hom_{k[\Delta]}(J_\lambda,J_b)=b,
\qquad
\dim\Hom_{k[\Delta]}(J_b,J_\lambda)=b.
\]
Summing over all blocks \(J_b\subset W'\) gives the cross contribution \(2m'\) to the
centralizer dimension. Equivalently, the corresponding cross contribution to the orbit
dimension is \\ \(2\lambda m'-2m'=2(\lambda-1)m'\). This proves the displayed formula for
\(D(W)\).

For a single block \(J_\lambda\), the centralizer has dimension \(\lambda\), so
\(D(J_\lambda)=\lambda^2-\lambda\). By \Cref{lem:blk},
\[
 2\Omega_2(J_\lambda)-D(J_\lambda)\ge2,
 \qquad
 2\Omega_3(J_\lambda)-D(J_\lambda)\ge2.
\]
For a trivial module \(T\), one has \(\Omega_2(T)=\Omega_3(T)=D(T)=0\). Now use
induction on the number of nontrivial Jordan blocks. From the decomposition
\[
 \Sym^2W=\Sym^2J_\lambda\oplus(J_\lambda\otimes W')\oplus\Sym^2W',
\]
we get
\[
\begin{aligned}
&\bigl(2\Omega_2(W)-D(W)\bigr)
-
\bigl(2\Omega_2(J_\lambda)-D(J_\lambda)\bigr)
-
\bigl(2\Omega_2(W')-D(W')\bigr) \\
&\qquad
=
2\Omega(J_\lambda\otimes W')-2(\lambda-1)m'.
\end{aligned}
\] By \eqref{eq:tenslb},
\[
\Omega(J_\lambda\otimes W')\ge \Omega(J_\lambda)\dim W'=(\lambda-1)m'.
\]
Thus the right-hand side is nonnegative. Induction gives \(2\Omega_2(W)-D(W)\ge2\).
Similarly, using the decomposition of \(\Sym^3W\), we have
\[
\begin{aligned}
&\bigl(2\Omega_3(W)-D(W)\bigr)
-
\bigl(2\Omega_3(J_\lambda)-D(J_\lambda)\bigr)
-
\bigl(2\Omega_3(W')-D(W')\bigr) \\
&\qquad
=
2\Omega(\Sym^2J_\lambda\otimes W')
+
2\Omega(J_\lambda\otimes\Sym^2W')
-
2(\lambda-1)m'.
\end{aligned}
\] Again by \eqref{eq:tenslb},
\[
\Omega(\Sym^2J_\lambda\otimes W')\ge \Omega_2(J_\lambda)m'\ge(\lambda-1)m',
\]
so the displayed cross contribution is nonnegative. Induction gives
\(2\Omega_3(W)-D(W)\ge2\). Moreover, in degree \(3\) the additional cross term
\(\Omega(J_\lambda\otimes \Sym^2W')\) is positive whenever \(W'\neq0\). Hence the
inequality is then strict beyond the basic lower bound. If \(W'=0\), then
\(W=J_\lambda\) is a single block; in the range \(m=\lambda\ge4\), \Cref{lem:blk} gives
\(
2\Omega_3(J_\lambda)-D(J_\lambda)\ge3.
\)
Therefore, for every nontrivial module \(W\) with \(m\ge4\), \(2\Omega_3(W)-D(W)\ge3\).
It remains to prove \eqref{eq:T}. For any finite-dimensional order-\(p\) unipotent
\(k[g]\)-module \(M\), define \(\Theta(M):=\Omega_2(M)+\Omega_3(M)-\dim M-D(M)\). Also
let \(\tau(M)\) denote the total dimension of the trivial Jordan blocks in \(M\);
equivalently, \(\tau(M)\) is the dimension of the largest direct summand of \(M\) on
which \(\Delta\) acts as zero.

We prove by induction on the number of nontrivial Jordan blocks that \(\Theta(M)\ge
-\tau(M)\) for every \(M\), with strict inequality whenever \(M\) is nontrivial and
\(\dim M\ge4\).

If \(M\) is trivial, then \(\Omega_2(M)=\Omega_3(M)=D(M)=0\), so \(\Theta(M)=-\dim
M=-\tau(M)\). If \(M=J_\lambda\) is one nontrivial block, then \(\tau(M)=0\), and
\Cref{lem:blk} gives
\[
 \Theta(J_\lambda)
 \ge
 \left\lfloor\frac{\lambda^2}{2}\right\rfloor
 +
 \frac{(\lambda-1)(\lambda+2)}2
 -
 \lambda
 -
 (\lambda^2-\lambda)
 \ge0.
\]
Moreover this inequality is strict for \(\lambda\ge4\).

Now suppose \(M\) is nontrivial and write \(M=J_\lambda\oplus M'\), where \(J_\lambda\)
is a largest nontrivial block and \(m':=\dim M'\). The cross contribution to
\(
\Theta(M)-\Theta(J_\lambda)-\Theta(M')
\)
is
\[
C:=
\Omega(J_\lambda\otimes M')
+
\Omega(\Sym^2J_\lambda\otimes M')
+
\Omega(J_\lambda\otimes\Sym^2M')
-
2(\lambda-1)m'.
\]
By \eqref{eq:tenslb},
\[
 C\ge
 \bigl(\Omega_2(J_\lambda)-(\lambda-1)\bigr)m'
 +
 (\lambda-1)\dim\Sym^2M'.
\]
Since
\(
\Omega_2(J_\lambda)-(\lambda-1)\ge1,
\)
this lower bound is at least \(m'\), hence at least \(\tau(M')\). If \(M'\neq0\), then
\(\dim\Sym^2M'>0\), so the lower bound is strictly larger than \(\tau(M')\).

By the induction hypothesis applied to \(M'\), \(\Theta(M')\ge-\tau(M')\). Also
\(\Theta(J_\lambda)\ge0\). Therefore
\(
\Theta(M)
=
\Theta(J_\lambda)+\Theta(M')+C
\ge0
\)
for every nontrivial \(M\). Moreover the inequality is strict if either \(M'\neq0\), or
if \(M'=0\) and \(M=J_\lambda\) has length \(\lambda\ge4\). Hence \(\Theta(M)>0\)
whenever \(M\) is nontrivial and \(\dim M\ge4\).

Applying this to the original module \(W\), we obtain \(\Theta(W)>0\). That is,\\
\(\Omega_2(W)+\Omega_3(W)-m-D(W)>0\). Since \(D(W)=D\), this is exactly
\(\Omega_2(W)+\Omega_3(W)-m>D\). The proof is complete.
\end{proof}

\subsection{Geometric applications}\label{sec:opgeom}

Recall that throughout this section
\[
D=\dim\Cl_{\GL(V)}(N),
\qquad
\Omega_d(V)=\dim\Img(\Delta|S_d).
\]
For a quotient \(R\), we similarly write \(\Omega_d(R):=\dim\Img(\Delta|R_d)\). The
estimates of \Cref{thm:opunif} will be applied to the \(k[g]\)-module \(V^\vee\), whose
dimension is \(m=\dim V\).

\begin{proposition}\label{prop:opeq}
Let \(X\subset\PP^n\) be a general smooth positive-dimensional complete intersection
whose lowest-degree block is \((d^r)\). Assume either
\[
 d=2,
 \quad r\ge3,
 \qquad\text{or}\qquad
 d\ge3,
 \quad r\ge2.
\]
Then \(X\) admits no nontrivial linear automorphism of order \(p\).
\end{proposition}

\begin{proof}
Fix a nontrivial order-\(p\) unipotent class. By \Cref{lem:fl}, any automorphism of
\(X\) preserves the lowest-degree equation space \(U\in\Gr(r,S_d)\). Thus \(U\) must be
\(\Delta\)-stable, and \eqref{eq:stgr} gives corresponding codimension at least
\(r\Omega_d(V)-r(r-1)\). It is enough to show that this is larger than \(D\).

If \(d=2\), then \(r\ge3\) and \(m=n+1\ge r+2\). By \eqref{eq:I2},
\(\Omega_2(V)\ge m\ge r+2\), and by \eqref{eq:2q}, \(2\Omega_2(V)-D\ge2\). Therefore
\[
\begin{aligned}
 r\Omega_2(V)-r(r-1)-D
 &= (r-2)\Omega_2(V)+\bigl(2\Omega_2(V)-D\bigr)-r(r-1) \\
 &\ge (r-2)(r+2)+2-r(r-1)
 = r-2>0.
\end{aligned}
\]

Now assume \(d\ge3\). By the monotonicity \eqref{eq:ambmonp}, \(\Omega_d(V)\ge
\Omega_3(V)\). If \(r=2\), then \eqref{eq:2cs} gives
\(
 2\Omega_d(V)-2-D\ge 2\Omega_3(V)-2-D>0.
\)
If \(r\ge3\), then \(m\ge r+2\). Using \(\Omega_d(V)\ge\Omega_3(V)\),
\eqref{eq:2cs}, and \eqref{eq:I3}, we get
\[
\begin{aligned}
 r\Omega_d(V)-r(r-1)-D
 &\ge (r-2)\Omega_3(V) +\bigl(2\Omega_3(V)-D\bigr)-r(r-1) \\
 &\ge (r-2)(2m-3)+3-r(r-1) \\
 &\ge (r-2)(2r+1)+3-r(r-1) \\
 &= (r-1)^2>0.
\end{aligned}
\]
Thus the stabilized locus associated with this conjugacy class is proper. Since there
are only finitely many such classes, the result follows.
\end{proof}

\begin{lemma}\label{lem:2q2g}
Assume that \(k\) is algebraically closed, \(\operatorname{char} k\neq 2\), and
\(n\ge 4\). Let \(Y\subset \mathbb P^n_k\) be a general smooth complete
intersection of two quadrics. Then its linear automorphism group
\(
\Aut_L(Y):=\{g\in \PGL_{n+1}(k)\mid g(Y)=Y\}
\)
is a finite \(2\)-group. In fact, for such a general \(Y\), one has
\(\Aut_L(Y)\cong (\mathbb Z/2\mathbb Z)^n\).
\end{lemma}

\begin{proof}
Let \(W=H^0(\mathbb P^n,\mathcal I_Y(2))\). Since \(Y\) is a complete
intersection of two quadrics, \(\dim W=2\), and \(\mathbb P(W)\) is precisely
the pencil of quadrics containing \(Y\).

If \(g\in \Aut_L(Y)\), then \(g^*W=W\). Hence \(g\) induces an automorphism of
the parameter line \(\mathbb P(W)\simeq \mathbb P^1\). This induced automorphism
preserves the discriminant locus \(\Delta\subset \mathbb P(W)\), namely the set
of singular quadrics in the pencil.

Since $Y$ is a general complete intersection, we may assume its associated pencil $\mathbb P(W)$ satisfies two generic conditions: 
(1) The discriminant locus $\Delta \subset \mathbb P(W)$ consists of exactly $n+1$ distinct points (i.e., the pencil is regular). 
(2) The set $\Delta$ has a trivial setwise stabilizer in $\PGL(W) \cong \PGL_2(k)$. 

The second condition relies on a classical fact from projective geometry: an automorphism of $\mathbb P^1$ is uniquely determined by its action on any three distinct points. Consequently, a general set of $m \ge 5$ points on $\mathbb P^1$ possesses no non-trivial projective symmetries. Since $n \ge 4$, our locus $\Delta$ consists of $n+1 \ge 5$ general points, so any projective transformation preserving the set $\Delta$ must be the identity.

Recall that any $g \in \Aut_L(Y)$ induces an automorphism on $\mathbb P(W)$ that maps singular quadrics to singular quadrics, meaning it must preserve the set $\Delta$. Because the stabilizer of $\Delta$ is trivial, this induced automorphism is the identity. Therefore, $g$ fixes every point in $\mathbb P(W)$, which implies that $g$ preserves every individual quadric as a point of the pencil.

Since \(\operatorname{char} k\neq2\) and the discriminant is reduced, we may
diagonalize the pencil. Thus, after a projective change of coordinates, we may
write
\[
\begin{aligned}
Q_0&=\sum_{i=0}^n x_i^2, \quad Q_1=\sum_{i=0}^n \lambda_i x_i^2,
\end{aligned}
\]
where the \(\lambda_i\) are distinct. The singular members of the pencil are
\(Q_1-\lambda_iQ_0\), for \(i=0,\dots,n\), and the unique singular point of
\(Q_1-\lambda_iQ_0\) is the coordinate point \([e_i]\).

Since \(g\) fixes each singular member of the pencil, it fixes its unique
singular point. Hence \(g\) fixes each coordinate point
\([e_0],\dots,[e_n]\). Therefore \(g\) is represented by a diagonal matrix, say
\(x_i\mapsto a_i x_i\).

Moreover, since the induced action of \(g\) on \(\mathbb P(W)\) is trivial, it
fixes the point \([Q_0]\in \mathbb P(W)\). Thus \(g^*Q_0=cQ_0\) for some
\(c\in k^\times\). But \(g\) is diagonal, so this says
\(\sum a_i^2x_i^2=c\sum x_i^2\), hence \(a_i^2=c\) for every \(i\). After
multiplying the representing matrix by a common scalar, we may assume
\(a_i=\pm1\) for all \(i\).

Thus every element of \(\Aut_L(Y)\) is a projective sign change. Conversely, any
projective sign change \(x_i\mapsto \epsilon_i x_i\), with
\(\epsilon_i\in\{\pm1\}\), preserves both \(Q_0\) and \(Q_1\), hence preserves
\(Y\). The common sign change acts trivially in projective space, so
\[
\Aut_L(Y)\cong \{\pm1\}^{n+1}/\{\pm(1,\dots,1)\}
\cong (\mathbb Z/2\mathbb Z)^n.
\]
In particular, \(\Aut_L(Y)\) is a finite \(2\)-group.
\end{proof}

\begin{proposition}\label{prop:op22}
Let \(X\) be a general smooth positive-dimensional complete intersection whose
multidegree begins with \((2,2,b^s), \ b\ge3, s\ge1\). Then \(X\) admits no
nontrivial linear automorphism of order \(p\).
\end{proposition}

\begin{proof}
By \Cref{lem:fl}, any linear automorphism of \(X\) preserves the intersection \(Y\) of
the first two quadrics. For general \(Y\), \Cref{lem:2q2g} shows that \(\Aut_L(Y)\) is a
\(2\)-group. Since \(p\neq2\), it contains no nontrivial element of order \(p\).
\end{proof}

\begin{proposition}\label{prop:opmix}
Let \(X\) be a general smooth positive-dimensional complete intersection whose
multidegree begins with one of the following patterns:
\[
 (2,b^s)\quad (b\ge3),
 \qquad
 (a,b^s)\quad (3\le a<b),
\]
where \(s\ge1\). Then \(X\) admits no nontrivial linear automorphism of order \(p\).
\end{proposition}

\begin{proof}
Fix a nontrivial order-\(p\) unipotent class. The condition \(n-(1+s)\ge1\) gives \(s\le
m-3\).

First consider the type \((2,b^s)\). By \Cref{lem:fl}, the quadric line is preserved;
since the action is unipotent, a defining smooth quadric \(Q\) is fixed by \(g\). The
first contribution is \(\Omega_2(V)\). Set \(R=S/(Q)\). By \Cref{lem:new,prop:bt}, the
degree-\(b\) minimal equation space is a general \(s\)-plane in \(R_b\), and it must be
\(\Delta\)-stable. Let \(\Omega_b(R):=\dim\Img(\Delta|R_b)\). By \eqref{eq:qmonp} and
\eqref{eq:qcr2},
\[
 \Omega_b(R)\ge \Omega_3(R)\ge \Omega_3(V)-m.
\]
Using \eqref{eq:I3}, we get \(\Omega_b(R)\ge m-3\ge s\). Hence, by \eqref{eq:stgr}, the
degree-\(b\) equation space contributes at least \(s\Omega_b(R)-s(s-1)\). This is at least
\(\Omega_b(R)\), because
\[
s\Omega_b(R)-s(s-1)-\Omega_b(R)=(s-1)(\Omega_b(R)-s)\ge0.
\]
Therefore the total contribution is at least \(\Omega_2(V)+\Omega_3(V)-m>D\) by
\eqref{eq:T}.

Now consider the type \((a,b^s)\), with \(3\le a<b\). We work over the open locus where
the degree-\(a\) hypersurface \(F\) is smooth, hence irreducible. The first contribution
is \(\Omega_a(V)\). Set \(R=S/(F)\), and let \(\Omega_b(R):=\dim\Img(\Delta|R_b)\). By
\eqref{eq:qmonp} and \eqref{eq:qcre},
\[
 \Omega_b(R)\ge \Omega_a(R)\ge \Omega_a(V)-1.
\]
Since \(\Omega_a(V)\ge \Omega_3(V)\ge m\), we have \(\Omega_b(R)\ge m-1\ge s\), and
therefore by \eqref{eq:stgr},
\[
s\Omega_b(R)-s(s-1)-\Omega_b(R)=(s-1)(\Omega_b(R)-s)\ge0.
\]
Thus the degree-\(b\) condition contributes at least \(\Omega_b(R)\). The total contribution
is at least
\[
 \Omega_a(V)+\Omega_a(V)-1\ge 2\Omega_3(V)-1>D
\]
by \eqref{eq:2cs}. This proves the proposition.
\end{proof}

\begin{proof}[Proof of \Cref{thm:op}]
Let
\[
 2\le d_1\le\cdots\le d_c,
 \qquad c\ge2,
 \qquad n-c\ge1,
\]
and assume the multidegree is not \((2,2)\). We exclude each order-\(p\) unipotent
conjugacy class; there are only finitely many.

Let \((a^r)\) be the lowest-degree block. If \(r\ge2\), then \Cref{prop:opeq} applies,
except when the lowest block is exactly \((2,2)\). In that remaining case the
multidegree begins with \((2,2,b^s)\) for some \(b\ge3\), and \Cref{prop:op22} applies.

If \(r=1\), then the multidegree begins either with \((2,b^s)\), \(b\ge3\), or with
\((a,b^s)\), \(3\le a<b\). These are covered by \Cref{prop:opmix}. Therefore a general
complete intersection in the stated range has no nontrivial linear automorphism of order
\(p\).
\end{proof}

\section{Involutions in characteristic \texorpdfstring{$2$}{2}}
\label{sec:c2}

This section proves \Cref{thm:c2}. Throughout this section \(k\) is algebraically closed
of characteristic \(2\), and we write \(m:=\dim V=n+1\). In the range of \Cref{thm:c2},
namely \(c\ge2\) and \(n-c\ge1\), one has \(m\ge4\).

Let \(g\in\PGL(V)\) be a nontrivial element of order \(2\). Choose a lift to \(\GL(V)\).
After rescaling this lift, we may assume \(g^2=I\). Since \(\operatorname{char}k=2\), we
may write \(g=I+N, \ N\neq0, \ N^2=0\). On \(S=\Sym(V^\vee)\) we set
\(\Delta:=g^*-\id\). Then \(\Delta^2=0\). We denote \(\kappa:=\rk(N)\). Thus \(1\le
\kappa\le \frac m2\). For each \(d\ge0\), set \(\Omega_d(V):=\dim\Img(\Delta|S_d)\).

\subsection{Linear algebra in characteristic \texorpdfstring{$2$}{2}}

\begin{lemma}\label{lem:c2orb}
For a nontrivial involution in characteristic \(2\) with \(\rk(N)=\kappa\),
\begin{equation}\label{eq:c2orb}
\dim \Cl_{\PGL(V)}(g)=2\kappa(m-\kappa).
\end{equation}
\end{lemma}

\begin{proof}
The Jordan type of \(N\) is \((2^\kappa,1^{m-2\kappa})\). The transpose partition is
\((m-\kappa,\kappa)\). Hence the centralizer of \(g\) in \(\GL(V)\) has dimension
\((m-\kappa)^2+\kappa^2\). Therefore
\[
\dim \Cl_{\GL(V)}(g)
=
m^2-\bigl((m-\kappa)^2+\kappa^2\bigr)
=
2\kappa(m-\kappa).
\]
This proves the claim.
\end{proof}

\begin{lemma}\label{lem:c2qimg}
For a nontrivial involution in characteristic \(2\) with \(\rk(N)=\kappa\),
\begin{equation}\label{eq:c2om2}
\Omega_2(V)=\kappa(m-\kappa).
\end{equation}
\end{lemma}

\begin{proof}
Choose a basis of \(V^\vee\) of the form
\[
x_1,\dots,x_\kappa,\ y_1,\dots,y_\kappa,\ z_1,\dots,z_q,
\qquad
q=m-2\kappa,
\]
such that
\[
\Delta(x_i)=y_i,
\qquad
\Delta(y_i)=0,
\qquad
\Delta(z_a)=0.
\]
Set
\[
M_i:=\Span(x_i,y_i),
\qquad
T:=\Span(z_1,\dots,z_q).
\]
Then
\[
V^\vee=\bigoplus_{i=1}^{\kappa}M_i\oplus T.
\]
Accordingly,
\[
\Sym^2(V^\vee)
=
\bigoplus_i\Sym^2(M_i)
\oplus
\bigoplus_{i<j}M_i\otimes M_j
\oplus
\bigoplus_i M_i\otimes T
\oplus
\Sym^2(T).
\]

We compute the image of \(\Delta\) on each summand. Since
\[
\Delta(fg)=\Delta(f)g+f\Delta(g)+\Delta(f)\Delta(g)
\]
for linear forms \(f,g\), we have on \(\Sym^2(M_i)\):
\[
\Delta(x_i^2)=y_i^2,\qquad
\Delta(x_i y_i)=y_i^2,\qquad
\Delta(y_i^2)=0.
\]
Thus the image is \(\Span(y_i^2)\), and has dimension \(1\).

For \(i<j\), on \(M_i\otimes M_j\) we have
\[
\Delta(x_i x_j)
=
x_i y_j+y_i x_j+y_i y_j,
\]
\[
\Delta(x_i y_j)=y_i y_j,\qquad
\Delta(y_i x_j)=y_i y_j,\qquad
\Delta(y_i y_j)=0.
\]
Hence the image is spanned by
\[
y_i y_j
\quad\text{and}\quad
x_i y_j+y_i x_j+y_i y_j,
\]
which are linearly independent. Thus the image dimension is \(2\).

On \(M_i\otimes T\), for \(1\le a\le q\),
\[
\Delta(x_i z_a)=y_i z_a,\qquad
\Delta(y_i z_a)=0.
\]
Therefore the image is \(\Span(y_i z_1,\dots,y_i z_q)\), and has dimension \(q\).
Finally, \(\Delta\) vanishes on \(\Sym^2(T)\), so the image dimension there is \(0\).

Adding these contributions gives
\[
\Omega_2(V)
=
\kappa
+
2\binom{\kappa}{2}
+
\kappa q
=
\kappa+\kappa(\kappa-1)+\kappa(m-2\kappa).
\]
Since \(q=m-2\kappa\), this simplifies to \(\Omega_2(V) = \kappa(m-\kappa)\).
\end{proof}

\begin{lemma}\label{lem:c2r1}
For a nontrivial involution in characteristic \(2\) acting on \(V\), with \(m=\dim
V\ge4\), one has
\begin{equation}\label{eq:c2om2lb}
\Omega_2(V)\ge m-1
\end{equation}
and
\begin{equation}\label{eq:c2om3lb}
\Omega_3(V)\ge 1+\binom m2.
\end{equation}
\end{lemma}

\begin{proof}
The bound for \(\Omega_2(V)\) follows from \eqref{eq:c2om2}:
\(\Omega_2(V)=\kappa(m-\kappa)\). Since \(1\le \kappa\le \frac m2\), the function
\(\kappa(m-\kappa)\) is increasing on the interval \([1,m/2]\). Hence \(\Omega_2(V)\ge
m-1\). It remains to prove the lower bound for \(\Omega_3(V)\). Choose one Jordan block
\(
M=\Span(x,y),
\
\Delta x=y,
\
\Delta y=0,
\)
and choose a \(g\)-stable complement \(W\), so that \(V^\vee=M\oplus W\), \(\dim
W=m-2\). The degree-three piece decomposes as a direct sum of \(g\)-stable subspaces:
\[
S_3
=
\Sym^3M
\oplus
\Sym^2M\otimes W
\oplus
M\otimes\Sym^2W
\oplus
\Sym^3W.
\]
Thus the image dimension on \(S_3\) is at least the sum of the lower bounds on these
three displayed summands involving \(M\).

First, \(\dim\Img(\Delta|\Sym^3M)=2\). Indeed,
\[
\Delta(x^3)=x^2y+xy^2+y^3,
\qquad
\Delta(x^2y)=y^3,
\]
and these two image vectors are independent.

Next consider \(\Sym^2M\otimes W\). For \(w\in W\),
\[
\Delta(x^2w)
=(x+y)^2g(w)-x^2w
=x^2\Delta(w)+y^2g(w).
\]
After projecting to the direct summand \(y^2W\), this image is \(g(w)\). The map
\(w\mapsto g(w)\) is injective, so this summand contributes at least \(\dim W=m-2\) to the
image.

Finally, on \(M\otimes\Sym^2W\), for \(f\in\Sym^2W\),
\[
\Delta(xf)
=(x+y)g(f)-xf
=x\Delta(f)+yg(f).
\]
After projecting to \(y\Sym^2W\), this image is \(g(f)\). Since \(g\) acts invertibly on
\(\Sym^2W\), this summand contributes at least
\(\dim\Sym^2W=\binom{m-1}{2}\). Adding these three contributions gives
\(
\Omega_3(V)
\ge
2+(m-2)+\binom{m-1}{2}
=
1+\binom m2.
\)
\end{proof}

\begin{lemma}\label{lem:c2mon}
For every \(d\ge1\),
\begin{equation}\label{eq:c2mon}
\Omega_d(V)\le \Omega_{d+1}(V).
\end{equation}
\end{lemma}

\begin{proof}
Choose a nonzero invariant linear form \(z\in \Ker(\Delta|S_1)\). Multiplication by
\(z\) gives an injective map \(S_d\hookrightarrow S_{d+1}\) commuting with \(\Delta\).
Therefore it sends \(\Img(\Delta|S_d)\) injectively into \(\Img(\Delta|S_{d+1})\), and
the claim follows.
\end{proof}

We shall also use the following refined Grassmannian estimate in the pure quadratic
case.

\begin{lemma}\label{lem:c2sqgr}
Let \(M\) be a finite-dimensional vector space with a square-zero endomorphism
\(\Delta\). Set
\[
E:=\dim M,
\qquad
\Omega_M:=\dim\Img\Delta.
\]
For \(s\ge1\), the locus \(\{U\in \Gr(s,M):\Delta U\subset U\}\) has dimension at most
\begin{equation}\label{eq:c2sqdim}
\max_{0\le a\le \lfloor s/2\rfloor}
\left\{
a(\Omega_M-a)+(s-a)(E-\Omega_M-s+a)
\right\}.
\end{equation}
Equivalently, its codimension in \(\Gr(s,M)\) is at least
\begin{equation}\label{eq:c2sqcod}
\min_{0\le a\le \lfloor s/2\rfloor}
\left\{
s(E-s)-a(\Omega_M-a)-(s-a)(E-\Omega_M-s+a)
\right\}.
\end{equation}
\end{lemma}

\begin{proof}
For a stable \(s\)-plane \(U\), set
\[
a:=\rk(\Delta|U),
\qquad
A:=\Delta(U),
\qquad
B:=U\cap\Ker\Delta.
\]
Then
\(
\dim A=a,
\qquad
\dim B=s-a,
\qquad
A\subset B\subset\Ker\Delta.
\)
In particular \(0\le a\le \lfloor s/2\rfloor\).

First choose \(A\subset\Img\Delta\), contributing \(a(\Omega_M-a)\) dimensions. Then
choose \(B\subset\Ker\Delta\) of dimension \(s-a\) containing \(A\). Equivalently, choose
\(B/A\) of dimension \(s-2a\) in \(\Ker\Delta/A\), whose dimension is
\(E-\Omega_M-a\); this contributes
\[
(s-2a)\bigl((E-\Omega_M-a)-(s-2a)\bigr)
=(s-2a)(E-\Omega_M-s+a)
\]
dimensions. Finally, for fixed \(A\) and \(B\), the choices of \(U\) are sections of the
induced surjection \(\Delta^{-1}(A)/B\longrightarrow A\). Its kernel has dimension
\[
\dim(\Ker\Delta/B)=(E-\Omega_M)-(s-a)=E-\Omega_M-s+a,
\]
so these sections contribute \(a(E-\Omega_M-s+a)\) dimensions. Adding these three
contributions gives
\[
a(\Omega_M-a)+(s-a)(E-\Omega_M-s+a).
\]
This proves the dimension bound, and the codimension bound follows by subtracting from
\linebreak\(\dim\Gr(s,M)\)\allowbreak \(=s(E-s)\).
\end{proof}

\begin{lemma}\label{lem:c2pqest}
Let \(M=S_2\), and let \(\Delta\) be induced by a nontrivial involution on \(V\). Let
\(s\ge3\), and assume \(m\ge s+2\). Then the locus of \(\Delta\)-stable \(s\)-planes in
\(S_2\) has codimension strictly larger than
\(
2\kappa(m-\kappa)=\dim\Cl_{\PGL(V)}(g).
\)
\end{lemma}

\begin{proof}
By \eqref{eq:c2sqcod}, the codimension is bounded below by the minimum, over \(0\le
a\le\lfloor s/2\rfloor\), of
\[
s(E-s)-a(\Omega_M-a)-(s-a)(E-\Omega_M-s+a),
\]
where
\(
E=\binom{m+1}{2},
\qquad
\Omega_M=\Omega_2(V)=\kappa(m-\kappa)
\)
by \eqref{eq:c2om2}. Expanding the expression in \eqref{eq:c2sqcod} gives
\[
\begin{aligned}
&s(E-s)-a(\Omega_M-a)-(s-a)(E-\Omega_M-s+a) \\
&\quad =sE-s^2-a\Omega_M+a^2 \\
&\qquad -(s-a)E+(s-a)\Omega_M \\
&\qquad +(s-a)s-a(s-a) \\
&\quad =s\Omega_M+a(E-2\Omega_M-2s)+2a^2.
\end{aligned}
\]
Thus it is enough to prove, for every \(0\le a\le s/2\), that
\[
s\Omega_M+a(E-2\Omega_M-2s)+2a^2>2\Omega_M.
\]
Equivalently, we need
\[
(s-2)\Omega_M+a(E-2\Omega_M-2s)+2a^2>0.
\]
Set \(B:=E-2\Omega_M-2s\).

If \(B\ge0\), then
\[
(s-2)\Omega_M+aB+2a^2\ge (s-2)\Omega_M>0,
\]
because \(s\ge3\) and \(\Omega_M>0\).

Now assume \(B<0\). Since \(0\le a\le s/2\), we have \(aB\ge \frac{s}{2}B\). Therefore
\[
\begin{aligned}
(s-2)\Omega_M+aB+2a^2
&\ge
(s-2)\Omega_M+\frac{s}{2}(E-2\Omega_M-2s) \\
&=
\frac{sE}{2}-2\Omega_M-s^2.
\end{aligned}
\]
Using
\(
\Omega_M=\kappa(m-\kappa)\le \frac{m^2}{4},
\)
we obtain
\[
\begin{aligned}
\frac{sE}{2}-2\Omega_M-s^2
&\ge
\frac{s}{2}\binom{m+1}{2}-\frac{m^2}{2}-s^2 \\
&=
\frac{(s-2)m^2+sm-4s^2}{4}.
\end{aligned}
\]
For fixed \(s\ge3\), the numerator \((s-2)m^2+sm-4s^2\) is increasing in \(m>0\). Since
\(m\ge s+2\), it is at least its value at \(m=s+2\), namely
\[
(s-2)(s+2)^2+s(s+2)-4s^2
=
s^3-s^2-2s-8.
\]
This is positive for every \(s\ge3\): for \(s=3\) it equals \(4\), and
\[
\bigl((s+1)^3-(s+1)^2-2(s+1)-8\bigr)-\bigl(s^3-s^2-2s-8\bigr)
=3s^2+s-2>0
\]
for \(s\ge3\). Hence
\(
(s-2)\Omega_M+a(E-2\Omega_M-2s)+2a^2>0.
\)
Therefore the codimension is strictly larger than
\(
2\Omega_M=2\kappa(m-\kappa)=\dim\Cl_{\PGL(V)}(g).
\)
\end{proof}

\subsection{Quadratic initial blocks}

\begin{proposition}\label{prop:c2pq}
Let \(X\subset \PP^n_k\) be a general smooth complete intersection whose lowest-degree
block is \((2^r)\), with \(r\ge3,  n-r\ge1\). Then \(X\) admits no nontrivial
linear automorphism of order \(2\).
\end{proposition}

\begin{proof}
By \Cref{lem:fl}, any linear automorphism of \(X\) preserves the lowest-degree equation
space \(U\in\Gr(r,S_2)\). If \(g\) is a nontrivial involution, then \(U\) must be
\(\Delta\)-stable. Since \(m=n+1\ge r+2\), \Cref{lem:c2pqest} shows that the locus of
such \(U\)'s has codimension strictly larger than \(\dim\Cl_{\PGL(V)}(g)\). Therefore
the stabilized locus for this conjugacy class is proper. Since there are only finitely
many involution classes, the claim follows.
\end{proof}

\begin{lemma}\label{lem:c2qq}
Assume \(m\ge4\). Let \(Q\in S_2\) be a smooth \(\Delta\)-invariant quadric, and set
\(R:=S/(Q)\). Then, for every \(d\ge3\),
\begin{equation}\label{eq:c2qq}
\dim\Img(\Delta|R_d)
\ge
\Omega_3(V)-m+\kappa.
\end{equation}
Consequently,
\begin{equation}\label{eq:c2qqcr}
\dim\Img(\Delta|R_d)
\ge
\frac{m^2-3m+4}{2}.
\end{equation}
Moreover,
\begin{equation}\label{eq:c2qqs}
\dim\Img(\Delta|R_d)>\Omega_2(V).
\end{equation}
\end{lemma}

\begin{proof}
First consider degree \(3\). Set \(C:=QS_1\subset S_3\). Since \(Q\) is
\(\Delta\)-invariant, \(C\) is \(\Delta\)-stable, and multiplication by \(Q\) identifies
\(S_1\) equivariantly with \(C\). Therefore
\[
\dim\Ker(\Delta|C)=\dim\Ker(\Delta|S_1)=m-\kappa.
\]

The image of \(\Delta\) on \(R_3=S_3/C\) has dimension
\[
\dim\Img(\Delta|R_3)
=
\Omega_3(V)-\dim\bigl(\Img(\Delta|S_3)\cap C\bigr).
\]
Since \(\Delta^2=0\), one has \(\Img(\Delta|S_3)\subset\Ker(\Delta|S_3)\). Hence
\(\Img(\Delta|S_3)\cap C\subset\Ker(\Delta|C)\). It follows that
\[
\dim\Img(\Delta|R_3)
\ge
\Omega_3(V)-(m-\kappa).
\]

Now pass to higher degrees. Since \(Q\) is smooth, \(R=S/(Q)\) is a domain. Choose a
nonzero invariant linear form \(z\in\Ker(\Delta|S_1)\). Its image in \(R_1\) is nonzero,
and multiplication by \(z\) gives injective maps \(R_e\hookrightarrow R_{e+1}\). Because
\(z\) is invariant, these maps commute with \(\Delta\). Hence
\(
\dim\Img(\Delta|R_e)\le \dim\Img(\Delta|R_{e+1})
\)
for every \(e\ge1\). Therefore, for every \(d\ge3\),
\[
\dim\Img(\Delta|R_d)
\ge
\dim\Img(\Delta|R_3)
\ge
\Omega_3(V)-m+\kappa.
\]

By \eqref{eq:c2om3lb}, \(\Omega_3(V)\ge 1+\binom m2\). Since \(\kappa\ge1\), we get
\[
\dim\Img(\Delta|R_d)
\ge
\left(1+\binom m2\right)-m+1
=
\frac{m^2-3m+4}{2}.
\]

It remains to prove the final assertion. If \(m\ge5\), then
\[
\frac{m^2-3m+4}{2}-\frac{m^2}{4}
=
\frac{(m-2)(m-4)}4>0,
\]
and \(\kappa(m-\kappa)\le m^2/4\). Hence, by \eqref{eq:c2om2},
\[
\frac{m^2-3m+4}{2}>\frac{m^2}{4}\ge \kappa(m-\kappa)=\Omega_2(V).
\]
If \(m=4\), then \(\kappa=1\) or \(\kappa=2\). By
\eqref{eq:c2om3lb},
\[
\dim\Img(\Delta|R_d)
\ge
\left(1+\binom42\right)-4+\kappa
=3+\kappa.
\]
For \(\kappa=1\), this is \(4>3=\Omega_2(V)\), while for \(\kappa=2\), this is
\(5>4=\Omega_2(V)\). Thus \[\dim\Img(\Delta|R_d)>\Omega_2(V)\ for\ all \ m\ge4.\]
\end{proof}

\begin{proposition}\label{prop:c2mixq}
Let \(X\subset \PP^n_k\) be a general smooth complete intersection whose multidegree
begins with
\[
(2,b^s),
\qquad
b\ge3,
\qquad
s\ge1,
\qquad
n-(1+s)\ge1.
\]
Then \(X\) admits no nontrivial linear automorphism of order \(2\).
\end{proposition}

\begin{proof}
It is enough to work over the nonempty open locus where the unique
quadratic minimal equation defines a smooth quadric. Let this quadric be defined by
\(Q\in S_2\). The condition \(n-(1+s)\ge1\) gives \(s\le m-3\). By \Cref{lem:fl}, the
line \(\langle Q\rangle\) is preserved by every linear automorphism of \(X\). Since an
order-\(2\) unipotent action on a one-dimensional vector space is trivial, \(Q\) is
\(\Delta\)-invariant.

The first equation contributes codimension \(\Omega_2(V)=\kappa(m-\kappa)\), by
\eqref{eq:c2om2}. Set
\[
R:=S/(Q),
\qquad
\Omega_b(R):=\dim\Img(\Delta|R_b).
\]
By \Cref{lem:new,prop:bt}, the next minimal equation space is a general \(s\)-plane in
\(R_b\), and it must be \(\Delta\)-stable. By \eqref{eq:c2qqs},
\(\Omega_b(R)>\Omega_2(V)\). Moreover \(\Omega_b(R)\ge s\). Indeed, if \(m\ge5\), then
\[
\Omega_b(R)\ge \frac{m^2-3m+4}{2}\ge m-3\ge s,
\]
where the first inequality is \eqref{eq:c2qqcr} and the middle inequality follows from
\((m^2-3m+4)/2-(m-3)=(m^2-5m+10)/2>0\). If \(m=4\), then \(s\le1\) and
the inequality \(\Omega_b(R)>\Omega_2(V)\) implies \(\Omega_b(R)\ge1=s\).

By \eqref{eq:stgr}, the stable-locus condition for the degree-\(b\) \(s\)-plane has
codimension at least \(s\Omega_b(R)-s(s-1)\). Since \(\Omega_b(R)\ge s\),
\[
s\Omega_b(R)-s(s-1)-\Omega_b(R)=(s-1)(\Omega_b(R)-s)\ge0,
\]
so this is at least \(\Omega_b(R)\). Therefore the total codimension is strictly larger
than
\[
\Omega_2(V)+\Omega_2(V)
=
2\kappa(m-\kappa)
=
\dim\Cl_{\PGL(V)}(g),
\]
where the last equality is \eqref{eq:c2orb}. Thus the stabilized locus for this
conjugacy class is proper.
\end{proof}

\begin{proposition}\label{prop:c222}
Let \(X\subset \PP^n_k\) be a general smooth complete intersection whose multidegree
begins with
\[
(2,2,b^s),
\qquad
b\ge3,
\qquad
s\ge1,
\qquad
n-(2+s)\ge1.
\]
Then \(X\) admits no nontrivial linear automorphism of order \(2\).
\end{proposition}

\begin{proof}
It is enough to work over the nonempty open locus where the
two-dimensional quadratic equation space cuts out a smooth complete
intersection of two quadrics. Let
\(
\Lambda\subset S_2
\)
be this quadratic equation space, and set
\(
Y:=V(\Lambda)\subset \PP^n.
\)
By \Cref{lem:fl}, every linear automorphism of \(X\) preserves
\(\Lambda\), and hence preserves \(Y\).

By the finiteness statement recalled in the incidence setup,
\(\Aut_L(Y)\) is finite. By \Cref{prop:bt}, after fixing \(Y\), the
degree-\(b\) minimal equation spaces giving smooth extensions contain a
nonempty open subset of
\[
\Gr\bigl(s,R_b(Y)\bigr),
\qquad
R_b(Y):=H^0(Y,\mathcal O_Y(b)).
\]
Moreover,
\[
s<\dim R_b(Y).
\]
Indeed, \(Y\) is a positive-dimensional smooth complete intersection,
hence integral. If \(0\neq z\in R_1(Y)\), multiplication by \(z^{b-1}\)
gives an injection
\(
R_1(Y)\hookrightarrow R_b(Y),
\)
and therefore\\
\(
\dim R_b(Y)\ge n+1>s.
\)

Let \(T(Y)\) be the finite set of nontrivial involutions in
\(\Aut_L(Y)\). We claim that every \(g\in T(Y)\) acts non-scalarly on
\(R_b(Y)\). Write
\(
g=1+\Delta,
\
\Delta^2=0.
\)
Since \(g\) is nontrivial in \(\PGL(V)\), its action on
\(R_1(Y)\simeq H^0(Y,\mathcal O_Y(1))\) is nontrivial. Choose \(u\) with
\(
v:=\Delta(u)\neq0.
\)
Then \(\Delta(v)=0\), and hence
\(
\Delta(uv^{b-1})=v^b.
\)
Since \(R(Y)\) is a domain, \(v^b\neq0\). Thus \(\Delta\) is nonzero on
\(R_b(Y)\), so \(g\) does not act as a scalar there.

By \Cref{lem:kill}, for every fixed \(Y\), a general
\(s\)-plane in \(R_b(Y)\) is stabilized by no element of \(T(Y)\).
Finally, the relative linear automorphism incidence is finite over the
smooth two-quadric locus. Hence the corresponding invariant-subspace
incidence in the relative Grassmannian has dimension strictly smaller
than that of the whole relative Grassmannian. Together with
\Cref{prop:bt}, this shows that a general smooth extension admits no
nontrivial involution.
\end{proof}

\subsection{Nonquadratic initial blocks}

\begin{proposition}\label{prop:c2eqnq}
Let \(X\subset \PP^n_k\) be a general smooth complete intersection whose lowest-degree
block is \((a^r)\), with \(a\ge3,  r\ge2, n-r\ge1\). Then \(X\) admits no
nontrivial linear automorphism of order \(2\).
\end{proposition}

\begin{proof}
By \Cref{lem:fl}, any linear automorphism of \(X\) preserves the lowest-degree equation
space \(U\in\Gr(r,S_a)\). If \(g\) is a nontrivial involution, then \(U\) must be
\(\Delta\)-stable. By \eqref{eq:stgr}, the codimension of the stable locus is at least
\(r\Omega_a(V)-r(r-1)\). By \eqref{eq:c2mon} and \eqref{eq:c2om3lb},
\[
\Omega_a(V)\ge \Omega_3(V)\ge 1+\binom m2.
\]
For \(2\le r\le m-2\), the function \(f(r):=r\Omega_a(V)-r(r-1)\) is increasing in
\(r\). Indeed,
\[
f(r+1)-f(r)=\Omega_a(V)-2r
\ge \left(1+\binom m2\right)-2(m-2)>0
\]
for \(m\ge4\). Hence it is minimized at \(r=2\). Therefore
\[
r\Omega_a(V)-r(r-1)
\ge
2\Omega_3(V)-2
\ge
m(m-1).
\]
On the other hand,
\(
\dim\Cl_{\PGL(V)}(g)=2\kappa(m-\kappa)\le \frac{m^2}{2}
\)
by \eqref{eq:c2orb}. Since \(m\ge4\), \(m(m-1)>\frac{m^2}{2}\). Thus the stabilized
locus for this conjugacy class is proper.
\end{proof}

\begin{lemma}\label{lem:c2hyp}
Let \(F\in S_a\), \(a\ge3\), be a nonzero \(\Delta\)-invariant irreducible form, and set
\(R:=S/(F)\). Then, for every \(b>a\),
\begin{equation}\label{eq:c2hyp}
\dim\Img(\Delta|R_b)\ge \Omega_a(V)-1.
\end{equation}
\end{lemma}

\begin{proof}
Since \(F\) is irreducible, \(R\) is a domain. Choose a nonzero invariant linear form\\
\(z\in\Ker(\Delta|S_1)\). The image of \(z\) in \(R_1\) is nonzero, because \(a\ge3\).
Hence multiplication by \(z\) gives injective maps \(R_t\hookrightarrow R_{t+1}\)
commuting with \(\Delta\). Therefore the image ranks on \(R_t\) are nondecreasing.

In degree \(a\), the quotient map
\(
S_a\longrightarrow R_a=S_a/\langle F\rangle
\)
can reduce the image rank by at most \(1\). Hence
\[
\dim\Img(\Delta|R_b)
\ge
\dim\Img(\Delta|R_a)
\ge
\Omega_a(V)-1.
\]
\end{proof}

\begin{proposition}\label{prop:c2mixnq}
Let \(X\subset \PP^n_k\) be a general smooth complete intersection whose multidegree
begins with
\[
(a,b^s),
\qquad
3\le a<b,
\qquad
s\ge1,
\qquad
n-(1+s)\ge1.
\]
Then \(X\) admits no nontrivial linear automorphism of order \(2\).
\end{proposition}

\begin{proof}
It is enough to work over the nonempty open locus where the unique
degree-\(a\) minimal equation defines a smooth hypersurface. Let this hypersurface be
defined by \(F\in S_a\). Then \(F\) is irreducible. The condition \(n-(1+s)\ge1\) gives
\(s\le m-3\). By \Cref{lem:fl}, the line \(\langle F\rangle\) is preserved by every
linear automorphism of \(X\). Since an order-\(2\) unipotent action on a one-dimensional
vector space is trivial, \(F\) is \(\Delta\)-invariant.

The first equation contributes codimension \(\Omega_a(V)\). Set
\[
R:=S/(F),
\qquad
\Omega_b(R):=\dim\Img(\Delta|R_b).
\]
By \Cref{lem:new,prop:bt}, the next minimal equation space is a general \(s\)-plane in
\(R_b\), and it must be \(\Delta\)-stable. By \eqref{eq:c2hyp}, \(\Omega_b(R)\ge
\Omega_a(V)-1\). Since, by \eqref{eq:c2mon} and \eqref{eq:c2om3lb},
\[
\Omega_a(V)\ge \Omega_3(V)\ge 1+\binom m2,
\]
we have \(\Omega_b(R)\ge \binom m2\ge s\). By \eqref{eq:stgr}, the stable-locus
condition for the degree-\(b\) \(s\)-plane has codimension at least
\(s\Omega_b(R)-s(s-1)\). Since \(\Omega_b(R)\ge s\),
\[
s\Omega_b(R)-s(s-1)-\Omega_b(R)=(s-1)(\Omega_b(R)-s)\ge0,
\]
so this is at least \(\Omega_b(R)\). Therefore the total codimension is at least
\[
\Omega_a(V)+\Omega_b(R)
\ge
2\Omega_a(V)-1
\ge
2\Omega_3(V)-1
\ge
m(m-1)+1.
\]
This is strictly larger than
\(
\dim\Cl_{\PGL(V)}(g)
=
2\kappa(m-\kappa)
\le
\frac{m^2}{2},
\)
where the equality is \eqref{eq:c2orb}. Thus the stabilized locus for this conjugacy
class is proper.
\end{proof}

\begin{proof}[Proof of \Cref{thm:c2}]
By
\Cref{prop:c2pq,prop:c222,prop:c2mixq,prop:c2eqnq,prop:c2mixnq},
all possible initial degree patterns are covered, except for the pure type
\((2,2)\), which is excluded by hypothesis. Hence a general complete
intersection in the stated range admits no nontrivial linear automorphism
of order \(2\).
\end{proof}
\section{Proof of the main theorem}\label{sec:proof}

\begin{proof}[Proof of \Cref{thm:main}]
By \Cref{prop:finite-reduction}, \(\Aut_L(X)\) is finite. Suppose it is
nontrivial. Then it contains an element of prime order \(\ell\).

If \(\operatorname{char}k=0\), or if
\(\operatorname{char}k=p>0\) and \(\ell\neq p\), this contradicts
\Cref{thm:ptp}. Thus it remains only to consider \(\ell=p\). If
\(p>2\), this is excluded by \Cref{thm:op}, while if \(p=2\), it is
excluded by \Cref{thm:c2}. Hence
\[
\Aut_L(X)=\{1\}.
\]
\end{proof}

\begin{proof}[Proof of Corollary~\ref{cor:lin}]
If \(\dim X\ge 3\), then by the Grothendieck--Lefschetz theorem the Picard group of a
smooth complete intersection is generated by the hyperplane class:\\ \(\Pic(X)\simeq
\ZZ\cdot \mathcal O_X(1)\). Hence every automorphism of \(X\) preserves \(\mathcal
O_X(1)\). Since \(X\) is linearly normal in \(\PP^n\), the restriction map
\[
H^0(\PP^n,\mathcal O_{\PP^n}(1))
\longrightarrow
H^0(X,\mathcal O_X(1))
\]
is an isomorphism. Therefore every automorphism of \(X\) is induced by a unique
projective linear transformation of the ambient projective space. Thus, in dimension at
least \(3\), one has \(\Aut(X)=\Aut_L(X)\).

For surfaces, by adjunction,
\[
K_X\simeq \mathcal O_X\Bigl(\sum_{i=1}^c d_i-n-1\Bigr).
\]
Since \(\Pic(X)\) is torsion-free, if this integer is nonzero, then preservation of the canonical class implies preservation of the hyperplane class. Hence the same argument shows that every automorphism is linear. The exceptional
surface case is the \(K3\) case, where \(\sum_{i=1}^c d_i=n+1\). In this case the
canonical bundle is trivial and no longer determines \(\mathcal O_X(1)\), so non-linear
automorphisms may occur and the comparison between \(\Aut(X)\) and \(\Aut_L(X)\) becomes
a separate issue.

Finally, in dimension \(1\), adjunction still gives
\[
K_X\simeq \mathcal O_X\Bigl(\sum_{i=1}^c d_i-n-1\Bigr).
\]
But the Picard group of a curve has a large degree-zero part, so the canonical class
does not generally determine the hyperplane class. Since \(\deg K_X=2g(X)-2\), the
comparison between \(\Aut(X)\) and \(\Aut_L(X)\) depends on the genus of the curve and
is more subtle.
\end{proof}

\end{document}